\documentclass [twoside,reqno,12pt] {amsart}
\usepackage[left=1in,right=1in,top=1in,bottom=1in]{geometry}

\usepackage[hidelinks]{hyperref}
\usepackage{amsfonts}
\usepackage{amssymb}

\usepackage{color}
\usepackage{graphics}
\usepackage{comment}

\newtheorem{thm}{Theorem}[section]

\newtheorem{lem}[thm]{Lemma}
\newtheorem{prop}[thm]{Proposition}

\newtheorem{defn}[thm]{Definition}
\newtheorem{rem}[thm]{Remark}

\newtheorem{assumption}{Assumption}

\theoremstyle{definition}

\numberwithin{equation}{section}

\renewcommand{\Re}{\mathrm{Re}}
\renewcommand{\Im}{\mathrm{Im}}

\newcommand{\C}{\mathbb{C}}

\renewcommand{\div}{\operatorname{div}}

\newcommand{\n}[1]{\langle #1 \rangle}

\newcommand{\R}{\mathbb{R}}

\newcommand{\pM}{{\partial M}}
\newcommand{\tr}{\hbox{tr} }
\newcommand\scl{\mathrm{scl}}

\def\hat{\widehat}
\def\tilde{\widetilde}
\def \bfo {\begin {eqnarray*} }
\def \efo {\end {eqnarray*} }
\def \ba {\begin {eqnarray*} }
\def \ea {\end {eqnarray*} }
\def \beq {\begin {eqnarray}}
\def \eeq {\end {eqnarray}}
\def \supp {\hbox{supp }}

\def \det {\hbox{det}}

\def \p {\partial}

\newcommand{\cO}{\mathcal{O}}
\newcommand{\cF}{\mathcal{F}}

\newcommand{\openbdy}{(0, T)\times \p M}

\newcommand{\conf}{c^{-\frac{n-2}{4}}}

\def\hat{\widehat}
\def\tilde{\widetilde}
\def \bfo {\begin {eqnarray*} }
\def \efo {\end {eqnarray*} }
\def \ba {\begin {eqnarray*} }
\def \ea {\end {eqnarray*} }
\def \beq {\begin {eqnarray}}
\def \eeq {\end {eqnarray}}
\def \supp {\hbox{supp }}

\def \det {\hbox{det}}

\def \p {\partial}

\title[Inverse problem for the convection-diffusion equation]{Recovery of time-dependent coefficients for the convection-diffusion equation on conformally transversally anisotropic manifolds from partial data}

\author[Liu]{Boya Liu}
\address{B. Liu, Department of Mathematics\\
North Dakota State University, Fargo\\
ND 58102, USA}
\email{boya.liu@ndsu.edu}

\author[Purohit]{Anamika Purohit}
\address{A. Purohit, Department of Mathematics\\
Indian Institute of Technology Bombay, Maharashtra, India}
\email{30007031@iitb.ac.in}

\begin{document}

\begin{abstract}
We study an inverse problem of recovering a time-dependent convection term and density coefficient of the convection-diffusion equation from partial data on a certain type of compact Riemannian manifold of dimension at least three. We prove that the knowledge of a partial input-output operator determines both coefficients uniquely up to a natural gauge. Our geometric setting is conformally transversally anisotropic manifolds, that is, compact Riemannian manifolds with boundary that are conformally embedded into a product of the Euclidean line and a transversal manifold. Additionally, we assume that the attenuated geodesic ray transforms of one-forms and functions are both injective on the transversal manifold.
\end{abstract}

\maketitle

\section{Introduction and Statement of Results}
\label{sec:intro}

Let $(M, g)$ be a smooth, compact, oriented Riemannian manifold of dimension $n \ge 3$ with smooth boundary $\p M$. We shall adopt the following notations throughout this paper: for any $0<T<\infty$,  we denote
$Q =(0,T)\times M^{\mathrm{int}}$ the spacetime, $\overline{Q}$ the closure of $Q$, and $\Sigma = \openbdy$ the lateral boundary of $Q$.

Let $A(t,x)=A_j(t,x)dx^j$ be a one-form with complex-valued coefficients, and let $q(t,x)$ be a complex-valued function. In what follows, we shall denote $d:C^\infty(Q)\to C^\infty(Q, T^\ast Q)$ to be the exterior derivative.
Its  $L^2$-adjoint $d^\ast$ is given by the following formula in local coordinates:
\begin{equation}
\label{eq:dstarv}
d^\ast v=-|g|^{-1/2}\p_{x_j}(|g|^{1/2}g^{jk}v_k),
\end{equation}
where $|g|=\det(g_{jk})$ and $v=v_jdx^j$.

Consider the convection-diffusion operator
\begin{equation}
\label{eq:convection_diffusion}
\begin{aligned}
\mathcal{L}_{c,g, A, q}
&=
c(x)^{-1} \p_t-\Delta_g -\langle A, d\cdot\rangle_g  + d^\ast (A\cdot) - \langle A, A \rangle_g+q,
\end{aligned}
\end{equation}
with time-dependent coefficients $A(t,x)$, called the \textit{convection term}, and $q(t,x)$, called the \textit{density coefficient}.  Here the Laplace-Beltrami operator $\Delta_g$ of the metric $g$ is given by the following expression   in local coordinates:
\[
\Delta_g v=|g|^{-1/2}\p_{x_j}\left(g^{jk}|g|^{1/2}\p_{x_k}v\right),
\]
where  $g^{jk}$ denotes the inverse of $g_{jk}$. Furthermore, $c(x)$ is a smooth and strictly positive function on $M$.  The convection-diffusion operator has extensive applications in  the transfer of mass or heat of different physical quantities, such as particles and energy, inside a physical system, due to convection and diffusion \cite{Stocker}. It also describes the velocity of a particle or price evolution \cite{Caro_Kian}, and  in chemical engineering \cite{Sahoo_Vashith}.

Motivated by these applications, in this paper we investigate the inverse problem of recovering the coefficients $A(t,x)$ and $q(t,x)$ on a class of Riemannian manifolds from partial boundary measurements. To formulate our main results, we shall make two geometric assumptions, the first of which is  that the Riemannian manifold $(M,g)$ is \textit{conformally transversally anisotropic}, defined as follows.
\begin{defn}
\label{def:CTA_manifolds}
A Riemannian manifold $(M,g)$ of dimension $n\ge 3$ with boundary $\p M$ is called conformally transversally anisotropic (CTA) if $M$ is a compact subset of a manifold $\R\times M_0^{\text{int}}$
with smooth boundary and nonempty interior, and $g= c(e \oplus g_0)$. Here $(\R,e)$ is the real line, $(M_0,g_0)$ is a smooth compact $(n-1)$-dimensional Riemannian manifold with smooth boundary, called the transversal manifold, and $c\in C^\infty(\R\times M_0)$ is a strictly positive function.
\end{defn}

Examples of CTA manifolds include precompact smooth proper subsets of Euclidean, spherical, and hyperbolic spaces, see  \cite{Ferreira_Kur_Las_Salo} for more examples of CTA manifolds.
We work on CTA manifolds because they admit a natural product structure, together with a limiting Carleman weight, both of which play a fundamental role in our analysis.
Indeed, we may express every point $x\in M$ as $x=(x_1,x')$, where $x_1 \in \R$ and $x'\in M_0$.
In particular, the projection $\varphi(x) = x_1$ is a limiting Carleman weight. The existence of such a function is equivalent to the existence of a parallel unit vector field for a conformal multiple of the metric $g$, and the converse holds when $M$ is simply connected, see \cite[Theorem 1.2]{Ferreira_Kenig_Salo_Uhlmann}. Locally, this condition is equivalent to   $(M,g)$ being conformal to the product of an interval and a Riemannian manifold $(M_0,g_0)$ of dimension $n-1$.

Consider the following initial boundary value problem:
\begin{equation}
\label{eq:ibvp}
\begin{cases}
\mathcal{L}_{c,g,A,q}u(t,x)=0 & \text{ in } Q,
\\
u(0,\cdot)=f_1(x) &\text{ in } M,
\\
u(t,x) = f_2(t,x)  &\text{ on } \Sigma.
\end{cases}
\end{equation}
If $(f_1, f_2) \in \mathcal{K}_0$, where the set $\mathcal{K}_0$ is given by
\[
\mathcal{K}_0 := \{(f|_{t=0}, f|_{\Sigma}): f\in H^1(0,T;H^{-1}(M)) \cap L^2(0,T;H^1(M))\},
\]
the problem \eqref{eq:ibvp} admits a unique solution $u \in H^1(0,T;H^{-1}(M)) \cap L^2(0,T;H^1(M))$, see for instance \cite{Caro_Kian,Evans,lionsNonHomogeneousBoundaryValue1972}.

We now introduce the boundary measurements considered in this paper. Let $\nu$ be the outward unit normal vector to $\pM$ with respect to the metric $g$. Throughout this paper, we set
$\p M_\pm = \{x \in \p M: \pm \p_\nu\varphi(x) \ge 0\}$
and denote $\Sigma_\pm=(0, T)\times \p M_\pm^{\text{int}}$ to be the corresponding subsets of the lateral boundary.
Furthermore, let $U',V'\subset \p M$ be neighborhoods of $\p M_+$, $\p M_-$, respectively. We also set $U=(0, T)\times U'$ and $V=(0,T)\times V'$. With these notations, we define the partial input-output operator by the formula
\begin{equation}
\label{eq:bdy_measurement}
\Lambda_{A,q}(f_1, f_2) = (u|_{t=T}, (\p_\nu + 2\n{A,\nu}_g)u|_{V}).
\end{equation}
The inverse problem considered in this paper is to uniquely determine the convection term $A(t,x)$ and the density coefficient $q(t,x)$ in the space-time $Q$ from $\Lambda_{A,q}$. However, due to  the gauge invariance of $\Lambda_{A,q}$, there is an obstruction  to the unique recovery of $A$ and $q$. More precisely, it is shown in \cite[Proposition 1.3]{Mishra_Purohit_Vashisth} that for any  function  $\Psi (t,x) \in W_0^{2,\infty}(Q)$, we have
\begin{equation}
\Lambda_{A,q} = \Lambda_{A-d \Psi,q-\p_t \Psi}.
\end{equation}
Therefore, one can only hope to recover the coefficients up to the gauge
\begin{equation}
\label{eq:gauge}
A^{(2)}=A^{(1)}-d \Psi, \quad q_2=q_1-\p_t \Psi.
\end{equation}
Moreover, it is worth emphasizing that measurements at $t=0$ and $t=T$ are not required for the global recovery of $A(t,x)$ and $q(t,x)$ in the Euclidean setting, which has been achieved in \cite{Sahoo_Vashith}. However,  such measurements are still needed in the setting of Riemannian manifolds. We refer readers to   \cite[Remark 1.5(1)]{Mishra_Purohit_Vashisth} for detailed discussions.

In addition to the assumption that $(M,g)$ is a CTA manifold, we need to assume that a certain geodesic ray transform is injective on the transversal manifold. Similar assumptions have been  adopted in the study of inverse problems on CTA manifolds, see for instance \cite{Cekic,Ferreira_Kur_Las_Salo,Krupchyk_Uhlmann_magschr,Liu_Saksala_Yan,Liu_Saksala_Yan_potential,Yan} and the references therein.

To proceed, let us introduce some definitions related to geodesic ray transforms on Riemannian manifolds with boundary.
We parametrize geodesics on $(M_0,g_0)$  by points on the unit sphere bundle $SM_0 = \{(x, \xi) \in TM_0: |\xi|=1\}$. Let
\[
\p_\pm SM_0 = \{(x, \xi) \in SM_0: x \in \p M_0, \pm \langle \xi, \nu(x) \rangle > 0\}
\]
be  the incoming (--) and outgoing (+) boundaries of $SM_0$, where $\langle \cdot, \cdot \rangle$ denotes  the Riemannian inner product of $(M_0,g_0)$.

Let $(x, \xi) \in \p_-SM_0$, and let $\gamma = \gamma_{x, \xi}(\tau)$ denote the unit speed geodesic in $M_0$ such that $\gamma(0) = x$ and $\dot{\gamma} (0) = \xi$. We use the notation  $\tau_{\mathrm{exit}}(x, \xi)>0$ to stand for the exit time of $\gamma$  and write $\tau_{\mathrm{exit}}(x, \xi) = +\infty$ if $\gamma$ does not exit $M_0$. Then  a unit speed geodesic segment $\gamma: [0, \tau_{\mathrm{exit}}(x, \xi)] \to M_0$, $0<\tau_{\mathrm{exit}}(x, \xi)<\infty$, is called \textit{non-tangential} if $\gamma(0)$,  $\gamma(\tau_{\mathrm{exit}}(x, \xi)) \in \pM_0$, the vectors $\dot{\gamma}(0)$ and $\dot{\gamma}(\tau_{\mathrm{exit}}(x, \xi))$ are transversal to $\p M_0$, and $\gamma(\tau)\in M_0^\mathrm{int}$ for all $\tau \in (0,\tau_{\mathrm{exit}}(x, \xi))$.

Suppose that the function $f\in C(M_0, \C)$, and the one-form $F\in C(M_0, T^\ast M_0)$. For a function $\alpha\in C^\infty(M_0)$ and every $(x, \xi) \in \p_-SM_0 \setminus \Gamma_-$, where $\Gamma_- = \{(x, \xi) \in \p_-SM_0: \tau_{\mathrm{exit}}(x, \xi)= +\infty\}$, we define the attenuated geodesic ray transform by
\begin{equation}
\label{eq:geo_trans_def}
I^\alpha(f, F)(x, \xi) = \int_0^{\tau_{\mathrm{exit}}(x, \xi)} [f(\gamma_{x, \xi}(t))+\n{F(\gamma_{x, \xi}(t)), \dot {\gamma}_{x, \xi}(t)}] \exp\left[\int_{0}^{t}\alpha(\gamma_{x, \xi}(s))ds\right] dt.
\end{equation}
It was established in \cite[Theorem 7.1]{Ferreira_Kenig_Salo_Uhlmann} that $I^\alpha$ is injective on simple manifolds.
Let us recall that a compact, simply connected Riemannian manifold with smooth boundary is said to be \textit{simple} if it has no conjugate points, and its boundary is strictly convex. We also remark that $I^\alpha$ is known to be injective on certain classes of non-simple manifolds, see \cite{deHoop_Ilmavirta,paternain2019geodesic}, among others.

Our second geometric assumption is as follows.
\begin{assumption}
\label{asu:inj}
There exists $\varepsilon>0$ such that for every function $\alpha\in C^\infty(M_0)$ satisfying $\|\alpha\|_{L^\infty(M_0)}<\varepsilon$,  if $I^\alpha(f, F)(x, \xi)=0$ for all $(x, \xi) \in \p_-SM_0 \setminus \Gamma_-$ such that $\gamma_{x, \xi}$ is a non-tangential  geodesic, then $f=\alpha p $ and  $F=dp$ in $M_0$ for some function $p \in C^1(M_0)$ with $p|_{\p M_0}=0$.
\end{assumption}

We are now ready to state the main result of this paper.
\begin{thm}
\label{thm:uniqueness}
Let  $(M,g)$ be a CTA manifold of dimension $n\ge 3$ with connected boundary $\p M$ such that Assumption \ref{asu:inj} holds for the transversal manifold $(M_0,g_0)$. For any $T>0$, let $Q=(0,T)\times M$. Assume  that  $A^{(i)}\in W^{1,\infty}(Q, T^\ast Q)$, $q_i \in C(\overline{Q})$, $i=1,2$, and that  $A^{(1)}=A^{(2)}$ and $q_1=q_2$ on $\p Q$. Then there exists a function $\Psi \in W^{2,\infty}_0(Q)$ such that
$\Lambda_{A^{(1)},q_1}=  \Lambda_{A^{(2)},q_2}$ implies  $A^{(2)}=A^{(1)}-d\Psi$  and $q_2=q_1-\p_t \Psi$.
\end{thm}

\begin{rem}
The assumption that $\p M$ is connected is used only to normalize the gauge function $\Psi$ in the recovery of the convection term. Indeed, we shall first prove that $A^{(1)}-A^{(2)}=d \Psi$ in $Q$. Since $A^{(1)}=A^{(2)}$ on $\p Q$, it follows immediately that $d \Psi=0$ on $\p Q$. Thus, for almost every $t\in (0,T)$, the function $\Psi(t,\cdot)$ is constant on each connected component of $\p M$. In particular, this constant is unique when $\p M$ is connected. After subtracting a time-dependent constant, we may assume that $\Psi=0$ on $\Sigma$.
\end{rem}

To the best of our knowledge, Theorem \ref{thm:uniqueness} is the first partial data uniqueness result for the time-dependent coefficients of the convection-diffusion operator on CTA manifolds with possibly non-simple transversal manifolds. In particular, it generalizes the result of \cite{Choulli_Kian}, in which only the recovery of the density coefficient $q(t,x)$ was established. From a geometric perspective, Theorem \ref{thm:uniqueness}   extends the partial data uniqueness result of \cite{Mishra_Purohit_Vashisth}, where the transversal manifold is assumed to be simple, to a broader geometric setting.

\subsection{Previous literature}

Similar to their hyperbolic counterparts, inverse problems for parabolic operators can be broadly divided into two categories according to whether the coefficients are time-independent or time-dependent. We first survey some results concerning time-independent coefficients. In  \cite{Isakov_91}, the uniqueness of   $q(x)$ from the corresponding Dirichlet-to-Neumann map was established by proving the density of products of solutions in some Lebesgue spaces, and the corresponding stability estimate was derived in \cite{Choulli_book}. Meanwhile, the same uniqueness result as in \cite{Isakov_91} was obtained in \cite{Avdonin_Seidman} by using the Boundary Control (BC) method.
Furthermore, a uniqueness result for the leading order term and the density coefficient from the corresponding Dirichlet-to-Neumann map was established in \cite{Canuto_Kavian}.
When the convection term $A(x)$ is also present, the author of \cite{Deng_Yu_Yang} showed that the data measured at the final time uniquely determines $A(x)$ in one dimension. In dimension two, assuming that  $q(x)=0$, uniqueness of $A(x)$ from a single boundary measurement was obtained in \cite{Cheng_Yamamoto}. In addition, both  uniqueness and logarithmic type stability estimates for $A(x)$ and $q(x)$ were established in \cite{Bellassoued_Rassas} in dimension three or higher.

Turning our attention to inverse problems involving time-dependent coefficients of the convection-diffusion equation, a stability result for the density coefficient $q(t,x)$ from full boundary data was established in a cylindrical domain in \cite{Gaitan_Kian}, and  a partial data result for $q(t,x)$ was obtained in \cite{Choulli_Kian}. Furthermore, the author of \cite{Feizmohammadi_heat} showed that the Dirichlet-to-Neumann map determines both the leading order term and the density coefficient. On the other hand, with the convection term $A(t,x)$, the unique recovery of $A(t,x)$ and $q(t,x)$ up to a natural gauge from the Dirichlet-to-Neumann map was obtained in \cite{Sahoo_Vashith} under a  smallness assumption on $A(t,x)$. This assumption was subsequently removed in the stability result \cite{Senapati_Vashisth}.  We also refer readers to additional partial data uniqueness results \cite{Kumar_Purohit,Purohit}, as well as stability estimates \cite{Bellassoued_Fraj,Bellassoued_Rassas}, which concern partial data settings different from the one considered in this paper. Let us note that  all of the aforementioned results were established in the Euclidean setting. In the Riemannian setting, to the best of our knowledge, the only uniqueness result concerning the  recovery of both $A(t,x)$ and $q(t,x)$ was obtained in \cite{Mishra_Purohit_Vashisth} for CTA manifolds with simple transversal manifolds.

\subsection{Main ideas of the proof}

The central step of the  proof is the construction of complex geometric optics (CGO) solutions to the equation $\mathcal{L}_{c ,g, A, q}u=0$ of the form
\[
u(t, x)=e^{\pm s^2 \beta^2 t \pm sx_1}(v_s(t,x)+r_s(t,x)), \quad (t, x)\in Q.
\]
Here $h\in (0,1)$ is a semiclassical parameter, $s=\frac{1}{h}$, $\beta\in (\frac{1}{\sqrt{3}}, 1)$ is a fixed parameter, $v_s$ is a Gaussian beam quasimode, and $r_s$ is a correction term that vanishes in a suitable sense as  $h \to 0$.  The spatial component $\pm x_1$ is a limiting Carleman weight, while $\pm s^2\beta^2t\pm sx_1$ is the associated $s$-dependent space-time phase used in the construction of CGO solutions.
The limiting Carleman weight is used to derive the necessary boundary and interior Carleman estimates, which were originally established in \cite{Mishra_Purohit_Vashisth}, see Proposition \ref{prop:bdy_Carleman} and Proposition \ref{prop:interior_Carleman}, respectively.
Specifically, the boundary Carleman estimate is used to control the solutions on the inaccessible part of the lateral boundary, whereas the purpose of the interior Carleman estimate is to establish the solvability result in Lemma \ref{lem:solvability}. This solvability result is subsequently used to prove the existence of the correction term $r_s$ in Theorem \ref{thm:existence_remainder}.
Let us remark that the Dirichlet data is prescribed  on the full lateral boundary in this paper, which is consistent with earlier works involving the convection term, see for instance  \cite{Mishra_Purohit_Vashisth,Sahoo_Vashith}, among others. To the best of our knowledge, recovering the coefficients $A$ and $q$ from the input-output operator defined in  \eqref{eq:bdy_measurement}, where the Dirichlet data is supported only on  an open subset of the lateral boundary, remains a challenging open problem.

We next discuss the construction of the quasimode $v_s$. Since the transversal manifold $(M_0,g_0)$ is not necessarily simple,  it is not possible to construct $v_s$ based on global geometric optics solutions. Instead, we construct Gaussian beam quasimodes for every non-tangential  geodesic in the transversal manifold $M_0$ in Theorem \ref{prop:Gaussian_beam} by using techniques from inverse problems for elliptic and hyperbolic operators, see for instance \cite{Cekic,Ferreira_Kur_Las_Salo,Krupchyk_Uhlmann_magschr,Liu_Saksala_Yan}. We then derive a concentration property for the quasimodes, which is given in
Theorem \ref{prop:limit_behavior}.  The construction of CGO solutions is completed in Theorem \ref{thm:existence_remainder}, where we establish the existence of the remainder term, as well as a suitable decay rate.

Compared to the aforementioned earlier works, a key distinction in the construction presented in this paper is that we take the parameter $s$ to be real. Indeed, the use of a complex parameter $s$ produces additional oscillatory factors once the CGO solutions are substituted into the  integral identity, and the limits of these terms cannot be directly determined as $h \to 0$. Our choice of $s$ eliminates these terms and allows us to pass to the limit in the subsequent analysis.

When the CGO solutions are substituted into the integral identity \eqref{eq:int_identity}, we first utilize the boundary Carleman estimate in Proposition \ref{prop:bdy_Carleman} to show that  the right-hand side of \eqref{eq:int_identity}
vanishes as $h\to 0$. On the other hand, thanks to the decay rate of the remainder terms \eqref{eq:est_remainders}, together with  the concentration property established in Theorem \ref{prop:limit_behavior},  the left-hand side of \eqref{eq:int_identity} converges to the attenuated geodesic ray transform of the Fourier transform of the convection term in the $x_1$-variable in the transversal manifold $(M_0,g_0)$ for almost every $t\in (0,T)$. We then utilize  Assumption \ref{asu:inj} to invert the attenuated geodesic ray transform and show that the spatial differential of the difference of the convection terms vanishes in $Q$. Afterwards, we show the existence of the function $\Psi(t,x)$ satisfying the first identity in \eqref{eq:gauge}. The proof is completed in two steps. We first provide a proof for the uniqueness result for the convection term $A(t,x)$, followed by showing the uniqueness for the density coefficient $q(t,x)$ up to the gauge.

This paper is organized as follows. We begin by stating the required boundary and interior Carleman estimates in Section \ref{sec:Carleman_estimate}. In Section \ref{sec:CGO_solution}, we construct the CGO solutions based on Gaussian beam quasimodes. Finally, we establish Theorem \ref{thm:uniqueness} in Section \ref{sec:proof}.

\subsection*{Acknowledgments}
The research of B.L. is partially supported by the  Simons Foundation Travel Support for Mathematicians (MPS-TSM-00013766). A.P. gratefully acknowledges the financial support provided by their project supervisor, Dr. Suman Kumar Sahoo, through the Sponsored Project ``Inverse Problems" (Project Code No. RD/0524-IRCCSH0-021), undertaken in the Department of Mathematics, Indian Institute of Technology Bombay.

\section{Carleman Estimates}
\label{sec:Carleman_estimate}

In this short section, we collect the  Carleman estimates that will be used for the construction of CGO solutions and the proof of Theorem \ref{thm:uniqueness}. These include both boundary and interior Carleman estimates, together with their immediate applications. The boundary Carleman estimate stated in Proposition  \ref{prop:bdy_Carleman} will be used to control boundary terms on portions of $\p Q$ where measurements are unavailable, while the interior estimate in Proposition \ref{prop:interior_Carleman} will be applied in Section \ref{sec:CGO_solution} to construct the remainder term in the CGO solutions. These results were originally established in the Euclidean setting in \cite{Sahoo_Vashith} and extended to Riemannian manifolds in \cite{Mishra_Purohit_Vashisth}. We state them for the convenience of readers.

Let $(M, g)$ be a CTA manifold as in Definition \ref{def:CTA_manifolds}, and let us denote $\tilde{g}=e\oplus g_0$. Since the conformal factor $c>0$ is time-independent, by  the conformal properties of the Laplace-Beltrami operator  \cite[Section 2]{Ferreira_Kur_Las_Salo}, we get
\begin{equation}
\label{eq:conformal_equivalence}
c^{\frac{n+2}{4}} (c^{-1}\p_t -\Delta_g)  (\conf u)= \p_t u -\Delta_{\tilde{g}}u- \big(c^{\frac{n+2}{4}}\Delta_g(\conf)\big)u.
\end{equation}
We also recall from \cite[Section 6]{Krupchyk_Uhlmann_magschr} that
\begin{equation}
\label{eq:conformal_first_order}
c^{\frac{n+2}{4}} \n{A,d(\conf u)}_g
=
\n{A,du}_{\tilde g}-\left(\frac{n-2}{4}c^{-1}\n{A,dc}_{\tilde{g}}\right) u,
\end{equation}
\begin{equation}
\label{eq:conformal_divergence}
d^\ast(A\conf u) = d^\ast_{\tilde g}(Au) - \left(\frac{n-2}{4}c^{-1}\n{A,dc}_{\tilde{g}}\right) u,
\end{equation}
and
\begin{equation}
\label{eq:conformal_zeroth_order}
c^{\frac{n+2}{4}} (\n{A,A}_g^2+q)(\conf u)
=
(\n{A,A}_{\tilde g}^2+cq)u.
\end{equation}
Hence, it follows from \eqref{eq:conformal_equivalence}--\eqref{eq:conformal_zeroth_order} that
\begin{equation}
\label{eq:equivalence_operator}
c^{\frac{n+2}{4}}\circ \mathcal{L}_{c,g,A,q} \circ \conf = \mathcal{L}_{\tilde{g}, A, \tilde{q}},
\end{equation}
where
\begin{equation}
\label{eq:equiv_coeff}
\tilde{q}=c(q-c^{\frac{n-2}{4}}\Delta_g(\conf)).
\end{equation}
Thus, by replacing the metric $g$ and the density coefficient $q$ with $\tilde{g}$ and $\tilde{q}$, respectively, we may assume that the conformal factor $c=1$. In what follows, we shall write  $\mathcal{L}_{g,A,q}$ to stand for the convection-diffusion operator $\mathcal{L}_{c,g,A,q}$ when $c=1$.

Let us first state the boundary Carleman estimate. We refer readers to   \cite[Theorem 2.1]{Mishra_Purohit_Vashisth} for its proof.
\begin{prop}
\label{prop:bdy_Carleman}
Let $(M,g)$ be a CTA manifold of dimension $n \ge 3$. Let $\beta \in (\frac{1}{\sqrt{3}},1)$, and let $\varphi(t,x) =  \frac{1}{h^2} \beta^2 t + \frac{1}{h} x_1$. Let $A \in W^{1,\infty}(Q,T^\ast Q)$ and $q\in C(\overline{Q}, \C)$. Then for all functions $u \in C^2(Q)$ such that $u|_{\Sigma} = u|_{t=0}=0$,
the following inequality holds for all $0<h \ll 1$:
\begin{equation}
\label{eq:boundary_Carleman}
\begin{aligned}
&\|e^{-\varphi} h^2\mathcal{L}_{g,A,q} u\|_{L^2(Q)}
+
\cO(h) \|e^{-\varphi(T,\cdot)}\nabla_g u(T,\cdot)\|_{L^2(M)}
+
\cO(h^{3/2}) \left(\int_{\Sigma_-} |\p_\nu \Gamma| |e^{-\varphi} \p_\nu u|^2 dS_g dt
\right)^{1/2}
\\
& \ge
\cO(h) \|e^{-\varphi}u\|_{L^2(Q)}
+
\cO(h^2)\|e^{-\varphi} \nabla_g u\|_{L^2(Q)}
+
\cO(h^2)\|e^{-\varphi(T,\cdot)}\nabla_g u(T,\cdot)\|_{L^2(M)}
\\
& \quad +
\cO(h^{3/2}) \left(\int_{\Sigma_+} \p_\nu \Gamma |e^{-\varphi} \p_\nu u|^2 dS_g dt\right)^{1/2}.
\end{aligned}
\end{equation}
Here $\Gamma(x)=x_1, \Sigma_\pm = (0,T)\times \p M_{\pm}^{\mathrm{int}}$, and $\p M_{\pm} = \{x\in \p M: \pm \p_\nu \varphi(x)\ge 0\}$.
\end{prop}

We next turn our attention to the interior Carleman estimate. To this end, we isometrically embed $(\overline{Q},e\oplus g)$  into a closed Riemannian manifold  $(N,g')$ without boundary, where $g' =e\oplus g$ in some open neighborhood $U \subset N$ of $\overline{Q}$. Here $U=(a,b)\times \hat M$, where $[0,T]\subset (a,b)$, and $\hat{M}$ is an open manifold such that $M \subset \hat M$.

Let $h>0$ be a small semiclassical parameter. We define the  space $H^1_{\mathrm{scl}}(N)$ as the completion of $C^\infty(N)$ with respect to the norm
\[
\|u\|_{H^1_{\mathrm{scl}}(N)}^2:=   \|u\|^2_{L^2(N)} + \|h\nabla_g u\|^2_{L^2(N)}.
\]
Furthermore, since $(N,g)$ is a Riemannian manifold without boundary, the space $H^{-1}_{\mathrm{scl}}(N)$ is the topological dual of $H^1_{\mathrm{scl}}(N)$ with respect to the norm
\[
\|v\|_{H^{-1}_{\mathrm{scl}}(N)} = \sup_{0\neq \psi \in H^{1}_{\scl}(N)}\frac{|\langle v,\psi \rangle_{L^2(N)}|}{\|\psi\|_{H^{1}_\mathrm{scl}(N)}}.
\]

We are now ready to state the following interior Carleman estimates for negative order Sobolev spaces, which were originally established in \cite[Lemma 2.2]{Mishra_Purohit_Vashisth}.

\begin{prop}
\label{prop:interior_Carleman}
Let $\varphi(t,x)$, $A(t,x)$, and $q(t,x)$ be the same as in Proposition \ref{prop:bdy_Carleman}. Then for all $0 <h \ll 1$ and $v\in C_0^\infty(Q)$, there exists a constant $C>0$, independent of $h$ and $v$, such that
\[
\|v\|_{L^2(0,T;L^2(N))}
\le
Ch^{-1}\|e^{\varphi} h^2 \mathcal{L}^\ast_{g,A,q}e^{- \varphi}v\|_{L^2(0,T;H^{-1}_{\mathrm{scl}}(N))},
\]
and
\[
\|v\|_{L^2(0,T;L^2(N))}
\le
Ch^{-1} \|e^{-\varphi} h^2 \mathcal{L}_{g,A,q}e^{  \varphi}v\|_{L^2(0,T;H^{-1}_{\mathrm{scl}}(N))}.
\]
\end{prop}

Armed with Proposition \ref{prop:interior_Carleman}, we obtain the following solvability result. The proof follows a standard argument involving the Hahn-Banach theorem and the Riesz representation theorem. We refer readers to \cite{Sahoo_Vashith} for details.

\begin{lem}
\label{lem:solvability}
Let $\varphi(t,x), A(t,x)$, and $q(t,x)$ be the same as in Proposition \ref{prop:bdy_Carleman}, and let $0 < h \ll 1$ be a  small semiclassical parameter. Then for any function  $F\in L^2(Q)$, there exists a solution $u\in H^1(0,T;H^{-1}(M))\cap L^2(0,T;H^1(M))$ to the equation $e^{\varphi} h^2\mathcal{L}^\ast_{g,A,q} e^{-\varphi} u = F$ in $Q$
such that
\begin{equation}
\label{eq:est_decreasing_solution}
\|u\|_{L^2(0,T;H^1_\mathrm{scl}(M))} \le \cO(h^{-1})\|F\|_{L^2(Q)}.
\end{equation}
On the other hand, there also exists a solution $v\in H^1(0,T;H^{-1}(M))\cap L^2(0,T;H^1(M))$ to the equation $e^{-\varphi} h^2\mathcal{L}_{g,A,q} e^{\varphi} v= F$ in $Q$
such that
\begin{equation}
\label{eq:est_increasing_solution}
\|v\|_{L^2(0,T;H^1_\mathrm{scl}(M))} \le \cO(h^{-1}) \|F\|_{L^2(Q)}.
\end{equation}
\end{lem}

\section{Construction of Complex Geometric Optics Solutions}
\label{sec:CGO_solution}

Let $(M,g)$ be a CTA manifold as in Definition \ref{def:CTA_manifolds}. Suppose that  $A(t,x)\in W^{1,\infty}(Q)$ and  $q\in C(\overline{Q})$. We set $0<h \ll 1$ to be a small semiclassical parameter and write $s=\frac{1}{h}$ throughout the rest of this paper. In the first part of this section, we again assume that the conformal factor $c=1$.
Our goal in this section is to construct an exponentially decaying solution to the equation $\mathcal{L}^\ast_{g,A,q} u_1=0$ in $Q$ of the form
\begin{equation}
\label{eq:form_exp_decaying_solution}
u_1(t,x) = e^{-s^2 \beta^2 t - s x_1}(v_s+r_1),
\end{equation}
as well as an exponentially growing solution to the equation $\mathcal{L}_{g,A,q} u_2=0$ in $Q$ that takes the form
\begin{equation}
\label{eq:form_exp_growing_solution}
u_2(t,x) = e^{s^2\beta^2 t + s x_1}(w_s+r_2).
\end{equation}
Here, $v_s$ and $w_s$ are smooth Gaussian beam quasimodes, and  the correction terms $r_1$ and $r_2$ vanish in a suitable sense as $h\to 0$.  Let us remark that, since $(M,g)$ is not necessarily simple, global constructions based on geodesic coordinates as in \cite{Mishra_Purohit_Vashisth} is not applicable to our case. Thus, we need to construct solutions  using Gaussian beams.

\subsection{Construction of Gaussian beam quasimodes}

In this subsection we focus on constructing Gaussian beam quasimodes, which was originally developed in \cite{Babich_Ulin, Ralston_1982} and has a broad range of applications in spectral theory and in microlocal analysis, see \cite{Babich_Buldyrev, Ralston_1977}, among others. On the other hand, starting with \cite{Belishev_Katchalov, Katchalov_Kurylev}, Gaussian beam quasimodes have been widely utilized  in the study of inverse problems. We refer readers to \cite{Cekic, Ferreira_Kenig_Salo_Uhlmann, Ferreira_Kur_Las_Salo, Feizmohammadi_et_all_2019, KKL_book, Krupchyk_Uhlmann_magschr, Liu_Saksala_Yan,Liu_Saksala_Yan_potential,Yan} and the references therein for application of Gaussian beams to elliptic and hyperbolic inverse problems.

Let $(M, g)$ be a CTA manifold with the conformal factor $c=1$, and let $T>0$.
Replacing the transversal manifold $(M_0, g_0)$ by a slightly larger manifold if necessary, we may assume without loss of generality that $(M, g)\subset (\R \times M_0^\mathrm{int}, e\oplus g_0)$.
The Gaussian beam quasimodes will be constructed in $\R^2 \times M_0^{\mathrm{int}}$.

We need to first regularize the convection term $A$ to obtain smooth Gaussian beam quasimodes. We shall explain why the regularization of $A$ is necessary in the proof of Theorem \ref{prop:Gaussian_beam}.
Let us  extend $A$ to a one-form in  $W^{1,\infty}(\R^2\times M_0^{\mathrm{int}})$ with compact support.
Then an application of  a partition of unity argument, in conjunction with a regularization in each coordinate patch, gives us the following result, see  \cite[Proposition 5.1]{Krupchyk_Uhlmann_magschr}.

\begin{prop}
\label{prop:regularization}
For any one-form $A \in W_0^{1,\infty}(\R^2\times M_0^{int})$, there exists an open  and bounded set $W\subset \R^2\times M_0^{int}$ and a family $A_\zeta\in C_0^\infty (W, T^\ast W)$ such that
\begin{equation}
\label{eq:est_diff_regu}
\begin{aligned}
&\|A-A_\zeta\|_{L^\infty}=o(1),
\quad
\|A_\zeta\|_{L^\infty}=\cO(1),
\quad
\|\nabla_g A_\zeta\|_{L^\infty}=\cO(\zeta^{-1}),
\\
&\|\p_t A_\zeta\|_{L^\infty}=o(\zeta^{-1}), \quad
\|\Delta_g A_\zeta\|_{L^\infty} =\cO(\zeta^{-2}), \quad \zeta \to 0.
\end{aligned}
\end{equation}
Here the $L^\infty$-norms are taken over the set $\R^2\times M_0^{int}$.
\end{prop}

Let $\gamma=\gamma(\tau)$ be a non-tangential geodesic in the transversal manifold $(M_0,g_0)$ of length $L>0$. By \cite[Example 9.32]{lee2012smooth}, we embed $(M_0, g_0)$ into a closed manifold $(\hat{M}_0, g_0)$ of the same dimension and extend $\gamma$ so that it remains a unit speed geodesic in $\hat{M}_0$. Furthermore, since $\gamma$ is non-tangential, there exists $\varepsilon>0$ such that $\gamma(\tau)\in \hat{M}_0\setminus M_0$, and $\gamma$ does not have self-intersections for any $\tau \in [-2\varepsilon, 0) \cup (L, L+2\varepsilon]$.

The main result of this subsection is as follows.

\begin{thm}
\label{prop:Gaussian_beam}
Let $(M, g)$ be a smooth CTA manifold with boundary. Let $s=\frac{1}{h} $, where $0<h\ll 1$ is a small semiclassical parameter, $T>0$, and $\beta \in (\frac{1}{\sqrt{3}},1)$.
Let $A \in W^{1,\infty}(Q)$ and $q \in C(\overline{Q})$. Then for each unit speed non-tangential geodesic $\gamma$ of $(M_0,g_0)$, there exist one parameter families of Gaussian beam quasimodes $v_s, w_s\in C^\infty(\R^2\times M_0)$ such that the following estimates hold as $h \to 0$:
\begin{equation}
\label{eq:estimate_v}
\|v_s\|_{H^1_\scl(Q)}=\cO(1),
\quad
\|e^{s^2 \beta^2 t + sx_1} h^2\mathcal{L}^{*}_{g, A, q} e^{-s^2 \beta^2 t - s x_1}v_s\|_{L^2(Q)} =o(h),
\end{equation}
and
\begin{equation}
\label{eq:estimate_w}
\|w_s\|_{H^1_\scl(Q)}=\cO(1),
\quad
\|e^{-s^2 \beta^2 t - s x_1} h^2\mathcal{L}_{g, A, q} e^{s^2 \beta^2 t + s x_1}w_s\|_{L^2(Q)}=o(h).
\end{equation}
\end{thm}

\begin{proof}
Let us start with the construction of Gaussian beam quasimodes for the conjugated operator $e^{s^2 \beta^2 t + s x_1} h^2\mathcal{L}^{*}_{A, q, g} e^{-s^2 \beta^2 t -  s x_1}$.  The proof mainly follows the ideas in \cite{Krupchyk_Uhlmann_magschr, Liu_Saksala_Yan} and modifies the arguments accordingly for the parabolic operator considered in this paper, see also \cite{Cekic, DDS_Kenig_Sjo_Uhl}.

Consider the ansatz $v_s$ of the form
\[
v_s(t,x_1, x';h) = e^{is\psi(x')} a(t,x_1,x';h),
\]
where the amplitude $a$ depends implicitly on the parameter $h$. Let us recall that the formal $L^2$-adjoint of the convection-diffusion operator $\mathcal{L}_{g, A, q}$ is given by
\begin{equation}
\label{eq:adjoint}
\mathcal{L}^{*}_{g, A, q} = - \p_t-\Delta_g +\n{\overline A,d\cdot}_g - d^\ast (\overline{A}\cdot)+ \langle \overline{A}, \overline{A} \rangle_g+\overline{q}.
\end{equation}
Due to the time-independence of the phase function $\psi$, it follows from direct computations that
\begin{equation}
\label{eq:computation_time_derivative}
e^{-is\psi} \p_t(e^{is\psi} a) = \p_t a.
\end{equation}
Also, since the metric $g=e\oplus g_0$, and $\psi$ is independent of $x_1$, we have
\begin{equation}
\label{eq:computation_Laplacian}
\begin{aligned}
e^{-is\psi}(-\Delta_g)(e^{is\psi} a)
&=
e^{-is\psi}(-\p_{x_1}^2-\Delta_{g_0})(e^{is\psi} a)
\\
&=
-\Delta_g a-is[2\langle \nabla_{g_0}\psi, \nabla_{g_0}a(x_1, \cdot)\rangle_{g_0}+(\Delta_{g_0}\psi)a]
+
s^2\langle \nabla_{g_0} \psi, \nabla_{g_0} \psi\rangle_{g_0}a.
\end{aligned}
\end{equation}
For the first order term in \eqref{eq:adjoint}, we compute directly to obtain
\begin{equation}
\label{eq:computation_first_order_1}
e^{-is\psi} \langle A, d(e^{is\psi} a)\rangle_g
=
\n{A, da}_g +is \n{A, \nabla_{g_0}\psi}_g a,
\end{equation}
and
\begin{equation}
\label{eq:computation_first_order_2}
e^{-is\psi} d^\ast (Ae^{is\psi} a)
=d^\ast (Aa) -is \n{A, \nabla_{g_0} \psi}_g a.
\end{equation}
Therefore, we deduce from \eqref{eq:computation_time_derivative}--\eqref{eq:computation_first_order_2} that
\begin{equation}
\label{eq:conjugated_L1}
\begin{aligned}
&e^{s^2\beta^2 t +s x_1} \mathcal{L}^{*}_{g, A, q} e^{-s^2 \beta^2 t - s x_1} \left( e^{is\psi(\tau, y)} a \right)
\\
= &
e^{is\psi} \left[ s^2 \left( \langle \nabla_{g_0} \psi, \nabla_{g_0} \psi \rangle_{g_0} - (1 - \beta^2) \right) a
\right.
\\
& \left.  +
s \left( 2 \p_{x_1} a - 2i \langle \nabla_{g_0} \psi, \nabla_{g_0} a \rangle_{g_0} - i (\Delta_{g_0} \psi) a -2 \overline A_1 a +2i  \langle \overline A(x_1, \cdot), \nabla_{g_0} \psi \rangle_{g_0} a \right) + \mathcal{L}^{*}_{g,A, q} a \right].
\end{aligned}
\end{equation}

We next construct the phase function $\psi$ and the amplitude $a$ such that they approximately satisfy the eikonal and transport equations corresponding to the coefficients of the $s^2$ and $s$ terms on the right-hand side of \eqref{eq:conjugated_L1}, respectively. The construction is lengthy and is carried out in several steps. We refer readers to \cite[Subsection 4.1.1]{Liu_Saksala_Yan} for a brief explanation of the purposes of each step.

\textit{Step 1: Fermi coordinates near $\gamma$}.
We fix a point $z_0=\gamma(\tau_0)$ on $\gamma([-\varepsilon, L+\varepsilon])$ and construct the quasimode locally near $z_0$. Let
\[
\Omega :=\{(\tau, y)\in \R \times \R^{n-2}: |\tau-\tau_0|<\delta, |y|<\delta'\}, \quad \delta, \delta'>0,
\]
be Fermi coordinates near $z_0$. We only describe the heuristic ideas of the construction of these coordinates and refer readers to \cite[Lemma 7.4]{Kenig_Salo} for details.
First, we choose a sufficiently small number $\delta>0$ such that the geodesic segment $\gamma|_{[\tau_0-\delta,\tau_0+\delta]}$  does not contain any self-intersections. Afterwards, we use parallel transport to choose some vector fields $E_1(\tau),\ldots,E_{n-2}(\tau)$ along $\gamma$ such that the set   $\{\dot\gamma(\tau),E_1(\tau),\ldots,E_{n-2}(\tau)\}$ spans a parallel orthonormal frame along $\gamma$. Our choice of $\delta$, combined with the inverse function theorem, implies that there exists $\delta'>0$ such that the map
$
F(\tau,y)=\exp_{\gamma(\tau)}\left(y^\alpha E_\alpha(\tau) \right)
$
is a diffeomorphism in $\Omega$, where $\exp$ is the exponential map of $(M_0,g_0)$, and $\alpha \in \{1,\dots,n-2\}$.

In these coordinates, the geodesic $\gamma$ near $z_0=\gamma(\tau_0)$ is represented  by the set $\rho=\{(\tau, 0): |\tau-\tau_0|<\delta\}$.
Moreover, the metric $g_0$ satisfies
\begin{equation}
\label{eq:g0_in_Fermi_coordinates}
g_0^{jk}(\tau,0)=\delta^{jk} \quad \text{and} \quad  \p_{y_l}g_0^{jk}(\tau, 0)=0.
\end{equation}
Thus, an application of Taylor's theorem yields that for small $|y|$ we have
\begin{equation}
\label{eq:g0_near_geo}
g_0^{jk}(\tau, y)=\delta^{jk}+\mathcal{O}(|y|^2).
\end{equation}
Furthermore, the Gaussian beam ansatz $v_s$ is of the form
\[
v_s(t, x_1, \tau, y)= e^{is\psi(\tau, y)} a(t, x_1, \tau, y; h).
\]
Our goal  is to find the phase function $\psi\in C^\infty(\Omega, \mathbb{C})$ such that \begin{equation}
\label{eq:property_of_phi}
\Im \psi \ge 0, \quad \Im \psi|_\rho=0, \quad \Im \psi(\tau, y)
\text{ is bi-Lipschitz equivalent to }
|y|^2,
\end{equation}
as well as an amplitude $a \in C^\infty(\mathbb{R}^2  \times \Omega, \mathbb{C})$ such that $\mathrm{supp}(a(t, x_1, \cdot)) \subset \{ |y| \leq \delta/2 \}$.

\textit{Step 2: Find the phase function}. We start the construction of the Gaussian beam quasimode by  finding a phase function $\psi$.

Similar to  \cite{Ferreira_Kur_Las_Salo,KKL_book,Krupchyk_Uhlmann_magschr,Liu_Saksala_Yan,Liu_Saksala_Yan_potential,Ralston_1977,Ralston_1982}, we find a function $\psi(\tau, y) \in C^\infty(\Omega, \C)$ that satisfies the equation
\begin{equation}
\label{eq:eikonel_eq}
\langle \nabla_{g_0} \psi, \nabla_{g_0}\psi \rangle_{g_0}-(1-\beta^2)
=
\mathcal{O}(|y|^3), \quad y\to 0,
\end{equation}
such that the inequality
\begin{equation}
\label{eq:imaginarypart_phi}
\Im \psi(\tau,y)\ge d|y|^2
\end{equation}
holds for some constant $d>0$ depending on $\beta$.

Arguing similarly to \cite[Subsection 4.1.3]{Liu_Saksala_Yan}, we obtain
\begin{equation}
\label{eq:phase_form}
\psi(\tau, y) =  \sqrt{1 - \beta^2} \left(\tau+ \frac{1}{2} H(\tau) y \cdot y\right).
\end{equation}
Here $H(\tau)$ is a smooth complex-valued symmetric matrix such that $\Im H(\tau)$ is positive-definite. Indeed, $H(\tau)$ is the unique solution to the initial value problem for the matrix Riccati equation
\begin{equation}
\label{eq:Riccati}
H'(\tau) + H(\tau)^2 = F(\tau), \quad H(\tau_0) = H_0,
\end{equation}
where $F(\tau)$ is a suitable symmetric matrix,

and $H_0$ is a  symmetric complex-valued matrix such that $\Im (H_0)$ is positive-definite, see \cite[Lemma 2.56]{KKL_book}.

\textit{Step 3: Find the amplitude.}
We next construct an amplitude $a$ of the form
\begin{equation}
a(t, x_1, \tau, y, h; \zeta) = h^{-\frac{n-2}{4}} a_0(t, x_1, \tau; \zeta) \chi\left( \frac{y}{\delta} \right),
\end{equation}
where $\zeta$ is a regularization parameter that we shall specify later. Also, $\chi\in C_0^\infty(\R^{n-2})$ is a cutoff function such that $\chi=1$ when $|y|\le \frac{1}{4}$, and $\chi =0$ for $|y| \ge \frac{1}{2}$.
Furthermore, the function $a_0 \in C^\infty(\mathbb{R}_t \times \mathbb{R}_{x_1} \times [\tau_0 - \delta, \tau_0 + \delta])$ is independent of $y$ and solves the approximate transport equation
\begin{equation}
\label{eq:a_0eqn}
2 \p_{x_1} a_0 - 2i \langle \nabla_{g_0} \psi, \nabla_{g_0} a_0 \rangle_{g_0} - i (\Delta_{g_0} \psi) a_0 -2 (\overline A_\zeta)_1 a_0 +2i \langle \overline A_\zeta(x_1, \cdot), \nabla_{g_0} \psi \rangle_{g_0} a_0 = \mathcal{O}(|y| \zeta^{-1})
\end{equation}
as $y, \zeta \to 0$. Here $A_\zeta$ is the regularized convection term.

We first compute
$\langle \nabla_{g_0} \psi, \nabla_{g_0} a_0 \rangle_{g_0}$.
To this end, we obtain from \eqref{eq:phase_form} that
\begin{equation}
\label{eq:dtau_phase}
\p_\tau \psi(\tau, y) = \sqrt{1 - \beta^2} +\cO(|y|^2).
\end{equation}
Therefore, we get from \eqref{eq:g0_near_geo} and \eqref{eq:dtau_phase} that
\begin{align*}
\langle \nabla_{g_0} \psi, \nabla_{g_0} a_0 \rangle_{g_0}
&= \left( \delta^{jk} + \cO(|y|^2) \right) \p_j \psi  \p_k a_0
\\
&= \sqrt{1 - \beta^2}   \p_\tau a_0 + \sqrt{1 - \beta^2}   H(\tau)   y \cdot \p_y a_0 + \cO(|y|^2)\p_\tau a_0+\cO(|y|)\p_ya_0.
\end{align*}

Turning our attention to computing $\Delta_{g_0} \psi$ near the geodesic $\gamma$, we utilize   \eqref{eq:g0_near_geo} and \eqref{eq:dtau_phase} again to obtain
\[
(\Delta_{g_0} \psi)(\tau, 0) = \sqrt{1 - \beta^2}   \tr H(\tau).
\]
This gives us
\begin{equation}
\label{eq:Laplacian_phase}
(\Delta_{g_0} \psi)(\tau, y) = \sqrt{1 - \beta^2}   \tr H(\tau) + \cO(|y|).
\end{equation}

We next Taylor expand the coefficients appearing on the left-hand side of \eqref{eq:a_0eqn}. Let us write
\begin{align*}
A_\zeta(t,x_1, \tau, y) = A_\zeta (t,x_1, \tau,0)+\int_{0}^{1}(\nabla_y A_\zeta(t,x_1,\tau, ys))yds.
\end{align*}
Thus, an application of \eqref{eq:est_diff_regu} yields that
\begin{equation}
\label{eq:expansion_A1}
A_\zeta(t,x_1, \tau, y) = A_\zeta (t,x_1, \tau, 0)+\mathcal{O}(|y|\zeta^{-1}).
\end{equation}
Furthermore, by writing $A=A_1dx_1+A_\tau d\tau+A_ydy$, as well as noting that
\[
\p_\tau \psi(\tau, y) = \sqrt{1-\beta^2} +\cO(|y|^2) \quad \text{ and } \quad \p_y \psi(\tau,y) = \cO(|y|),
\]
we get
\begin{equation}
\label{eq:expansion_inner_product}
\langle A_\zeta, \nabla_{g_0}\psi \rangle_{g_0}= \sqrt{1-\beta^2}(A_\zeta)_\tau (t, x_1, \tau, 0)+\mathcal{O}(|y|\zeta^{-1}).
\end{equation}
Hence, by combining  \eqref{eq:dtau_phase}--\eqref{eq:expansion_inner_product}, we deduce from the approximate transport  equation \eqref{eq:a_0eqn} that
\begin{align*}
2\left( \p_{x_1}
- i\sqrt{1-\beta^2} \p_\tau \right)a_0
-&
\left[
i\sqrt{1-\beta^2}\tr H(\tau)
+2
(\overline A_\zeta)_1 (t,x_1,\tau,0)
\right.
\\
& \left.
-2i \sqrt{1-\beta^2} (\overline A_\zeta)_\tau (t,x_1,\tau,0)
\right] a_0
=
\cO(|y| \zeta^{-1}).
\end{align*}
Therefore, we shall require that the function $a_0(t,x_1,\tau;\zeta)$ satisfies the equation
\begin{equation}
\label{eq:a_0eqn_after}
\begin{aligned}
&\left(\p_{x_1} - i \sqrt{1 - \beta^2} \p_\tau \right)a_0
\\
&= \frac{1}{2}\left[i \sqrt{1 - \beta^2}\tr H(\tau) +2 (\overline A_\zeta)_1 (t,x_1,\tau,0) -2i  \sqrt{1 - \beta^2} ( \overline A_\zeta)_\tau (t,x_1,\tau,0)\right]a_0.
\end{aligned}
\end{equation}

We next perform a change of variables to write the left-hand side of \eqref{eq:a_0eqn_after} as a $\p$-equation.

To that end, for any fixed number $\beta\in(\frac{1}{\sqrt3}, 1)$, let $S: \R \to \R$ be a linear function defined by $S(\tau) = \sqrt{1-\beta^2} \tau$. Clearly, its inverse function is given by $S^{-1}(\tau) = \frac{1}{\sqrt{1-\beta^2}}\tau =:r$. Thus, we obtain the equation
\begin{equation}
\label{eq:transport_after_change}
\begin{aligned}
&(\p_{x_1} - i \p_r) a_0'
\\
&= \frac{1}{2} \left[i \sqrt{1 - \beta^2} \tr H(\sqrt{1 - \beta^2} r) +2 (\overline A_{\zeta})'_1 (t,x_1, r, 0) -2i \sqrt{1 - \beta^2} (\overline A_{\zeta})'_\tau (t, x_1, r, 0) \right] a'_0,
\end{aligned}
\end{equation}
where $a'_0 = a_0 \circ S$, $(A_{\zeta})'_1  = (A_{\zeta})_1 \circ S$,
and $(A_{\zeta})'_\tau = \left(A_{\zeta}\right)_\tau \circ S$.

Let us now write $\p = \frac{1}{2}(\p_{x_1} - i \p_r)$ and look for a solution to  \eqref{eq:transport_after_change} in the form of
\begin{equation}
\label{eq:form_a0_prime}
a'_0 (t, x_1, r; \zeta) = e^{\Phi_{1,\zeta} (t, x_1, r) + f_1(r)}.
\end{equation}
Then it follows from direct computations that $\Phi_{1,\zeta}$ and $f_1$ must satisfy the equations
\begin{equation}
\label{eq:equation_Phi}
\p{\Phi_{1,\zeta}} (t, x_1, r) = \frac{1}{2} \left[ (\overline A_{\zeta})_1' (t,x_1, r, 0) -i  \sqrt{1 - \beta^2} (\overline A_{\zeta})'_\tau (t, x_1, r, 0) \right],
\end{equation}
and
\begin{equation}
\label{eq:equation_f1}
\p_{r} f_1(r) = -\frac{1}{2} \sqrt{1 - \beta^2} \tr H(\sqrt{1 - \beta^2} r).
\end{equation}
We obtain $f_1$ by integrating the right-hand side of \eqref{eq:equation_f1} with respect to the $r$-variable.
On the other hand, to solve the $\p$-equation \eqref{eq:equation_Phi}, we use the standard fundamental solution $E(x_1, r) = \frac{1}{\pi (x_1-i r)}$ of the operator $\p$ to obtain
\begin{equation}
\label{eq:fund_solution}
\Phi_{1,\zeta}(t,x_1,r) = \frac{1}{2} E\ast \left( (\overline A_{\zeta})_1'(t,x_1,\gamma(r)) - i \sqrt{1 - \beta^2} (\overline A_{\zeta})_\tau' (t, x_1, \gamma(r)) \right),
\end{equation}
see \cite[Section 5.4]{Friedlander_Joshi}.

We next verify that the smooth amplitude $a_0$ satisfies the approximate transport equation \eqref{eq:a_0eqn}. To this end, we estimate $a_0(\cdot;\zeta)$ and its first and second derivatives in the set $[0,T]\times J_{x_1} \times [r_0-\delta, r_0+\delta]$. Here $J_{x_1} \subset \R$ is an open and bounded interval containing the $x_1$-coordinates of all points in $Q$.

By Proposition \ref{prop:regularization}, the one-form $A_\zeta$ is supported in a bounded open subset of $\R^2 \times M_0^\mathrm{int}$. Thus, by arguing similarly as in \cite[Subsection 4.1.4]{Liu_Saksala_Yan}, there exists a compact set $K\subset \R^2$ such that the following estimates hold for every $(t, x_1, r) \in [0,T]\times J_{x_1} \times [0, \frac{1}{\sqrt{1-\beta^2}}]$:
\begin{equation}
\label{eq:est_Phi1}
\|\Phi_{1,\zeta}\|_{L^\infty} = \cO(1), \quad \|\nabla_{g_0} \Phi_{1,\zeta}\|_{L^\infty}, \|\p_{x_1} \Phi_{1,\zeta}\|_{L^\infty}=o(\zeta^{-1}), \quad \|\Delta_{g_0} \Phi_{1,\zeta}\|_{L^\infty}=o(\zeta^{-2}), \quad \zeta \to 0.
\end{equation}
Here the $L^\infty$-norms are measured over the set $[0,T]\times J_{x_1} \times [r_0-\delta, r_0+\delta]$.

We now set $\zeta = h^{\alpha}$, $\alpha \in (0,\frac{1}{2})$, and note that the coordinate change function $S$ is independent of $\zeta$. Therefore, by applying
\eqref{eq:est_diff_regu} and \eqref{eq:fund_solution}, in conjunction with the bound $\|f_1\|_{L^\infty}=\cO(1)$, we have the following estimates as $h\to 0$:
\begin{equation}
\label{eq:bound_deri_b0_h_old}
\begin{aligned}
&\|a_0(\cdot; h)\|_{L^\infty}=\cO(1),
\quad
\|\nabla_{g_0} a_0(\cdot; h)\|_{L^\infty},   \|\p_t a_0(\cdot; h)\|_{L^\infty},
\|\p_{x_1} a_0(\cdot; h)\|_{L^\infty} =o(h^{-\alpha}),
\\
&\|\Delta_{g_0} a_0(\cdot; h)\|_{L^\infty} =o(h^{-2\alpha}).
\end{aligned}
\end{equation}
Hence, by utilizing  \eqref{eq:dtau_phase}--\eqref{eq:a_0eqn_after}, together with \eqref{eq:bound_deri_b0_h_old}, we deduce from \eqref{eq:a_0eqn} that
\begin{align*}
& 2 \p_{x_1} a_0 - 2i \langle \nabla_{g_0} \psi, \nabla_{g_0} a_0 \rangle_{g_0} - i (\Delta_{g_0} \psi) a_0 -2 (\overline A_\zeta)_1 a_0 +2i \langle \overline A_\zeta(t, x_1, \cdot), \nabla_{g_0} \psi \rangle_{g_0} a_0
\\
&=  -2i \sqrt{1 - \beta^2}   H(\tau)   y \cdot \p_y a_0 +\cO(|y|^2)\p_\tau a_0+\cO(|y|^2)\p_ya_0 + \cO(|y|) + \cO(|y|h^{-\alpha})
\\
&=\cO(|y|h^{-\alpha}).
\end{align*}
This shows that the amplitude $a_0$ we have obtained in this step satisfies the approximate transport equation \eqref{eq:a_0eqn}.

\textit{Step 4: Verification of the estimate \eqref{eq:estimate_v}.}
Our next goal is to establish the estimates in  \eqref{eq:estimate_v} for the quasimode
\begin{equation}
\label{eq:beam_form_regu}
v_s(t, x_1, r, y;h)
=
e^{is\psi(\sqrt{1-\beta^2}r, y)}a'(t, x_1, r,y; h)
=
e^{is\psi(\sqrt{1-\beta^2}r, y)}h^{-\frac{n-2}{4}}a_0'(t, x_1, r; h)\chi(y/\delta')
\end{equation}
in the open subset $(0, T) \times J_{x_1} \times \Omega$ of $Q$, where $\Omega\subset M_0$ is the domain of Fermi coordinates near the point $z_0=\gamma(\tau_0)$. Let us first note that, due to the inequality \eqref{eq:imaginarypart_phi}, the following estimate holds  for any $k\geq 0$:
\begin{equation}
\label{eq:L2_power_y}
\begin{split}
\|h^{-\frac{n-2}{4}} |y|^k e^{-\frac{\Im \psi}{h} }\|_{L^2(|y| \le \delta'/2)}
&\leq
\|h^{-\frac{n-2}{4}} |y|^k e^{-\frac{d}{h}|y|^2}\|_{L^2(|y| \le \delta'/2)}
\\
& \leq  h^{k/2}
\bigg(\int_{\R^{n-2}} |z|^{2k} e^{-2d|z|^2}dz\bigg)^{1/2}
\\
&=\cO(h^{k/2}), \quad h \to 0.
\end{split}
\end{equation}
Here we have performed a change of variables $z= h^{-1/2}y$.

We are now ready to prove the estimates in \eqref{eq:estimate_v} locally near $z_0$. To establish the first one, we obtain from \eqref{eq:imaginarypart_phi}, \eqref{eq:bound_deri_b0_h_old}, and \eqref{eq:L2_power_y} with $k=0$ that
\begin{equation}
\label{eq:estimate_v_int}
\begin{aligned}
\|v_s\|_{L^2([0,T] \times \overline{J}_{x_1} \times \Omega)}
&\le
\|a_0'\|_{L^\infty([0,T] \times \overline{J}_{x_1} \times [r_0-\delta, r_0+\delta])}
\|e^{is\psi} h^{-\frac{n-2}{4}}\chi(y/\delta')\|_{L^2([0,T] \times \overline{J}_{x_1} \times \Omega)}
\\
&\le
\mathcal{O}(1) \|h^{-\frac{n-2}{4}}   e^{-\frac{d}{h} |y|^2}\|_{L^2(|y|\le \delta'/2)}
\\
&=\mathcal{O}(1), \quad h \to 0.
\end{aligned}
\end{equation}

We next show that
\begin{equation}
\label{eq:est_nablavs}
\|\nabla_g v_s\|_{L^2(Q)}= \cO(h^{-1}), \quad h\to 0.
\end{equation}
To this end, we compute directly from \eqref{eq:beam_form_regu} that
\[
\nabla_g v_s
=
h^{-\frac{n-2}{4}} e^{is\psi} \left[isa_0' \chi\left(\frac{y}{\delta'}\right) \nabla_g \psi + \chi\left(\frac{y}{\delta'}\right) \nabla_g a_0' + a_0' \frac{1}{\delta'} \nabla_{g_0} \chi\left(\frac{y}{\delta'}\right)\right].
\]
We observe that
\[
|e^{is\psi}|=e^{-\frac{1}{h} \Im \psi}.
\]
Since $e^{-\frac{1}{h}}=\cO(h^{\infty})$ as $h\to 0$, we deduce from \eqref{eq:imaginarypart_phi} that
$|e^{is\psi}|
\leq
e^{-\frac{\tilde d}{h}} $
for some constant $ \tilde d>0$ on $\supp (\nabla_{g_0}\chi(y/\delta'))$. Thus, an application of the estimates \eqref{eq:bound_deri_b0_h_old} and  \eqref{eq:L2_power_y} with $k=0$ gives us   \eqref{eq:est_nablavs}. Furthermore, we obtain the first estimate of  \eqref{eq:estimate_v} by combining \eqref{eq:estimate_v_int} and \eqref{eq:est_nablavs}.

We now move to verify the second estimate in \eqref{eq:estimate_v}. It requires us to estimate each term on the right-hand side of \eqref{eq:conjugated_L1}. For the $s^2$-term, it follows from
\eqref{eq:eikonel_eq},  \eqref{eq:imaginarypart_phi}, \eqref{eq:bound_deri_b0_h_old}, and \eqref{eq:L2_power_y} with $k=3$ that
\begin{equation}
\label{eq:est_first_term}
\begin{aligned}
&h^2\|e^{is\psi}s^2 (\langle \nabla_{g_0}\psi, \nabla_{g_0}\psi\rangle_{g_0} - (1 - \beta^2))a\|_{L^2([0,T] \times \overline{J}_{x_1} \times \Omega)}
\\
&=
h^2\|e^{is\psi} s^2 (\langle \nabla_{g_0}\psi, \nabla_{g_0}\psi\rangle_{g_0} - (1 - \beta^2)) h^{-\frac{n-2}{4}}a_0'\chi(y/\delta')\|_{L^2([0,T] \times \overline{J}_{x_1} \times \Omega)}
\\
&\le
\cO(1) \|a_0'\|_{L^\infty([0,T] \times \overline{J}_{x_1} \times [r_0-\delta, r_0+\delta])} \|h^{-\frac{n-2}{4}} |y|^3 e^{-\lambda d|y|^2}\|_{L^2(|y| \le \delta'/2)} =\cO(h^{3/2}), \quad  h \to 0.
\end{aligned}
\end{equation}

Turning our attention to the $s$-term,
by applying  the estimate  \eqref{eq:L2_power_y} with $k=1$, as well as the triangle inequality, we obtain
\begin{align*}
&h^2 \|e^{is\psi}s\left( 2 \p_{x_1} a - 2i \langle \nabla_{g_0} \psi, \nabla_{g_0} a \rangle_{g_0} - i (\Delta_{g_0} \psi) a  -2 (\overline A_\zeta)_1 a +2i \langle \overline A_\zeta(x_1, \cdot), \nabla_{g_0} \psi \rangle_{g_0} a \right) \|_{L^2([0, T]\times  \overline{J}_{x_1} \times \Omega)}
\\
&\le \cO(h) \|e^{is\psi} h^{-\frac{n-2}{4}} [|y|h^{-\alpha}\chi(y/\delta')-2i\n{\nabla_{g_0}\psi, \nabla_{g_0}\chi(y/\delta')}_{g_0}]\|_{L^2([0, T]\times  \overline{J}_{x_1} \times \Omega)}
\\
& \le \cO(h^{1-\alpha}) \|h^{-\frac{n-2}{4}} |y|  e^{-\frac{d}{h}|y|^2}\|_{L^2(|y| \le \delta'/2)} +\cO(e^{-\frac{\tilde{d}}{h}})
\\
&=\cO(h^{3/2-\alpha})
\\
&=o(h), \quad  h \to 0,
\end{align*}
where we have utilized the fact that $\alpha\in (0,\frac{1}{2})$ in the last step.

Let us note that the conjugated equation \eqref{eq:conjugated_L1} contains the convection term $A$, while its smooth approximation $A_\zeta$ appears in the estimate above. To account for this discrepancy, we apply the estimates \eqref{eq:est_diff_regu},  \eqref{eq:L2_power_y} with $k=0$, and the triangle inequality to deduce that
\begin{align*}
&h^2\|e^{is\psi} 2s [(\overline A_\zeta)_1-\overline A_1 + i \n{\overline  A(x_1, \cdot) - \overline A_\zeta (x_1, \cdot), \nabla_{g_0} \psi}_{g_0}]a\|_{L^2([0, T]\times  \overline{J}_{x_1} \times \Omega)}
\\
&\le \cO(h)\|A-A_\zeta\|_{L^\infty([0, T]\times  \overline{J}_{x_1} \times \Omega)} \|h^{-\frac{n-2}{4}} e^{-\frac{d}{h} |y|^2}\|_{L^2(|y| \le \delta'/2)}
\\
&=o(h),\quad h \to 0.
\end{align*}
Therefore, it follows from the previous two computations that
\begin{equation}
\label{eq:est_s_term}
\begin{aligned}
&h^2\|e^{is\psi}s\left( 2 \p_{x_1} a - 2i \langle \nabla_{g_0} \psi, \nabla_{g_0} a \rangle_{g_0} - i (\Delta_{g_0} \psi) a -2 \overline A_1 a +2i \langle \overline A(x_1, \cdot), \nabla_{g_0} \psi \rangle_{g_0} a \right) \|_{L^2([0, T]\times  \overline{J}_{x_1} \times \Omega)}
\\
&=
o(h), \quad h \to 0.
\end{aligned}
\end{equation}

It still remains to $\mathcal{L}^\ast_{g,A,q}a$ on the right-hand side of \eqref{eq:conjugated_L1}, where the operator $\mathcal{L}^\ast_{g,A,q}$ is given by \eqref{eq:adjoint}. First, for  the $\p_t$-term, due to the time-independence of the cut-off function $\chi$, it follows from the estimates \eqref{eq:bound_deri_b0_h_old} and \eqref{eq:L2_power_y} with $k=0$ that
\begin{equation}
\label{eq:est_dtsquare}
\begin{aligned}
h^2\|e^{is\psi}\p_t a\|_{L^2([0, T]\times  \overline{J}_{x_1} \times \Omega)}
&= \cO(h^2) \|e^{is\psi} h^{-\frac{n-2}{4}}\p_t a_0'\|_{L^2([0, T]\times  \overline{J}_{x_1} \times \Omega)}
\\
&=o(h^{2-\alpha})
\\
&= o(h^{3/2}), \quad h \to 0.
\end{aligned}
\end{equation}
We next estimate the term involving   $\Delta_g$. To this end, we apply the estimates   \eqref{eq:bound_deri_b0_h_old} and    \eqref{eq:L2_power_y} with $k=0$, as well as the triangle inequality, to get
\begin{equation}
\label{eq:est_laplacian}
\begin{aligned}
&h^2 \|e^{is\psi}(-\Delta_g a)\|_{L^2([0, T]\times  \overline{J}_{x_1} \times \Omega)}
\\
\le &
\cO(h^2) \|h^{-\frac{n-2}{4}} e^{is\psi}  \chi(y/\delta')\Delta_g a_0'\|_{L^2([0, T]\times  \overline{J}_{x_1} \times \Omega)}
\\
&+
\cO(h^2) \|h^{-\frac{n-2}{4}} e^{is\psi} [a_0' \Delta_g  \chi(y/\delta')+2\n{\nabla_g a_0', \nabla_g \chi(y/\delta')}_g]\|_{L^2([0, T]\times  \overline{J}_{x_1} \times \Omega)}
\\
\le &
\cO(h^2)  \|h^{-\frac{n-2}{4}} e^{-\frac{d}{h} |y|^2} h^{-2\alpha}\|_{L^2(|y| \le \delta'/2)} + \cO(e^{-\frac{\tilde{d}}{h}})
\\
= &\cO(h^{2-2\alpha})+\cO(e^{-\frac{\tilde{d}}{h}})
\\
= &
o(h), \quad h \to 0.
\end{aligned}
\end{equation}
To estimate the first order term in \eqref{eq:adjoint}, by the same computations as above, we have
\begin{equation}
\label{eq:est_first_order}
\begin{aligned}
&h^2 \|e^{is\psi}\n{\overline A, da}_g\|_{L^2([0, T]\times  \overline{J}_{x_1} \times \Omega)}
\\
&\le
\cO(h^2) \|e^{is\psi}h^{-\frac{n-2}{4}}\n{\overline A, da_0'}_g\|_{L^2([0, T]\times  \overline{J}_{x_1} \times \Omega)} +
\cO(h^2) \|e^{is\psi}h^{-\frac{n-2}{4}}\n{\overline A,a_0' d\chi(y/\delta')}_g\|_{L^2([0, T]\times  \overline{J}_{x_1} \times \Omega)}
\\
&\le
o(h^{2-\alpha}) + \cO(e^{-\frac{\tilde{d}}{h}})
\\
&=o(h^{3/2}), \quad h \to 0.
\end{aligned}
\end{equation}
Furthermore,  the same reasoning yields that
\begin{equation}
\label{eq:est_divergence}
\begin{aligned}
&h^2\|e^{is\psi}d^\ast(\overline Aa)\|_{L^2([0, T]\times  \overline{J}_{x_1} \times \Omega)}
\\
\le &
\cO(h^2)\|e^{is\psi} h^{-\frac{n-2}{4}} a_0 \chi(y/\delta') d^\ast \overline A\|_{L^2([0, T]\times  \overline{J}_{x_1} \times \Omega)}
+
\cO(h^2)\|e^{is\psi} h^{-\frac{n-2}{4}} \chi(y/\delta') \overline A\cdot \nabla_g a_0'\|_{L^2([0, T]\times  \overline{J}_{x_1} \times \Omega)}
\\
&+
\cO(h^2)\|e^{is\psi} h^{-\frac{n-2}{4}} a_0' \overline A\cdot \nabla_g \chi(y/\delta') \|_{L^2([0, T]\times  \overline{J}_{x_1} \times \Omega)}
\\
\le &
\cO(h^2) + o(h^{2-\alpha}) + \cO(e^{-\frac{\tilde{d}}{h}})
\\
= &
o(h^{3/2}), \quad h \to 0.
\end{aligned}
\end{equation}
Lastly, for the zeroth order term, we obtain from the estimate \eqref{eq:bound_deri_b0_h_old}  that
\begin{equation}
\label{eq:est_first_and_zero}
h^2\|e^{is\psi}(-\langle \overline A, \overline A \rangle_g+\overline{q})a\|_{L^2([0, T]\times  \overline{J}_{x_1} \times \Omega)}
=o(h^{2-\alpha})=o(h^{3/2}), \quad h \to 0.
\end{equation}
Hence, by combining the estimates \eqref{eq:est_dtsquare}--\eqref{eq:est_first_and_zero}, we get
\begin{equation}
\label{eq:est_op}
\|h^2\mathcal{L}^{*}_{g, A, q} a\|_{L^2([0, T]\times  \overline{J}_{x_1} \times \Omega)}=o(h), \quad h \to 0.
\end{equation}
Furthermore, due to the estimates \eqref{eq:est_first_term}, \eqref{eq:est_s_term}, and \eqref{eq:est_op}, we conclude from  \eqref{eq:conjugated_L1} that
\begin{equation}
\label{eq:est_op_regu}
\|e^{s^2 \beta^2 t + s x_1} h^2 \mathcal{L}^{*}_{g, A, q} e^{-s^2\beta^2 t - s x_1}v_s\|_{L^2([0, T]\times  \overline{J}_{x_1} \times \Omega)}=o(h), \quad h \to 0.
\end{equation}
The proof of the estimate \eqref{eq:estimate_v} locally on the set $[0, T]\times  \overline{J}_{x_1} \times \Omega$ is now complete.

We next derive an estimate for $\|v_s(t, x_1, \cdot)\|_{L^2(\pM_0)}$ for later purposes. By similar arguments as in \cite[Subsection 4.1.5]{Liu_Saksala_Yan}, as well as \cite[Section 5]{Krupchyk_Uhlmann_magschr}, we obtain
\begin{equation}
\label{eq:est_v_bdyM}
\begin{aligned}
\|v_s(t, x_1, \cdot)\|^2_{L^2(\pM_0 \cap \Omega)}  =\cO(1), \quad h \to 0.
\end{aligned}
\end{equation}

\textit{Step 5: Global construction.} Finally, we patch together the quasimodes constructed along small segments of the geodesic $\gamma$ to produce a global quasimode $v_s$ in
$\R^2\times M_0^{\mathrm{int}}$. We shall closely follow the arguments in \cite[Subsection 4.1.6]{Liu_Saksala_Yan} and present the proof for completeness.

Since $\hat{M}_0$ is compact and $\gamma(r):(-2\varepsilon, \frac{L}{\sqrt{1-\beta^2}}+2\varepsilon)\to \hat{M}_0$ is a non-tangential geodesic that does not form a loop,
due to  \cite[Lemma 7.2]{Kenig_Salo} and the choice of $\varepsilon>0$, the geodesic segment $\gamma|_{(-2\varepsilon, \frac{L}{\sqrt{1-\beta^2}}+2\varepsilon)}$
has finitely many self-intersection times $r_\ell\geq 0$, $\ell \in \{1,\ldots,R\}$, and
\[
-\varepsilon =r_0< r_1< \dots<r_R<r_{R+1} =\frac{L}{\sqrt{1-\beta^2}}+\varepsilon.
\]
By
\cite[Lemma 7.4]{Kenig_Salo},
there exists an open cover $\{(\Omega_\ell, \kappa_\ell)_{\ell=0}^{R+1}\}$ of the segment $\gamma([-\varepsilon, \frac{L}{\sqrt{1-\beta^2}}+\varepsilon])$ consisting of Fermi coordinate neighborhoods that satisfy the following properties:
\begin{enumerate}
\item[(1)] $ \kappa_\ell(\Omega_\ell)=I_\ell\times B$, where $I_\ell$ are open intervals and $B=B(0, \delta')$ is an open ball in $\R^{n-2}$. Here $\delta'>0$ can be taken arbitrarily small and the same for each $\Omega_\ell$.
\item[(2)] $ \kappa_\ell(\gamma(r))=(r, 0)$ for $r \in I_\ell$.
\item[(3)] $r_\ell$ only belongs to $I_\ell$ and $\overline{I_\ell}\cap \overline{I_k}=\emptyset$ unless $|\ell-k|\le 1$.
\item[(4)] $ \kappa_\ell= \kappa_k$ on $\kappa_\ell^{-1}((I_\ell\cap I_k)\times B)$.
\end{enumerate}
In particular,
the intervals $I_0$ and $I_{R+1}$ are chosen such that that they do not contain any self-intersection times.
On the other hand, if $\gamma$ has no self-intersections, then a single coordinate neighborhood of $\gamma|_{[-\varepsilon, \frac{L}{\sqrt{1-\beta^2}}+\varepsilon]}$ satisfies the properties (1) and (2).

We now move to construct $v_s$ globally. Let us first assume that $\gamma$ does not self-intersect at $r=0$. Arguing as in the earlier part of the proof, we obtain a quasimode
\[
v_s^{(0)}(t,x_1, r, y;h)=h^{-\frac{n-2}{4}}e^{is\psi^{(0)}(\sqrt{1-\beta^2}r, y)}e^{\Phi_{1, h}(t,x_1, r)+f_1(r)}\chi(y/\delta')
\]
in $\Omega_0$ with some fixed initial conditions at $r=-\varepsilon$ for the Riccati equation \eqref{eq:Riccati} that determines $\psi^{(0)}$.
We next choose $r_0'$ such that $\gamma(r_0')\in \Omega_0\cap \Omega_1$ and fix the initial conditions for
$\psi^{(1)}$ such that $\psi^{(1)}(r_0')=	\psi^{(0)}(r_0')$. This gives us a  quasimode in $\Omega_1$ of the form
\[
v_s^{(1)}(t,x_1, r, y;h)=h^{-\frac{n-2}{4}}e^{is\psi^{(1)}(\sqrt{1-\beta^2}r, y)}e^{\Phi_{1,h}(t,x_1, r)+f_1(r)}\chi(y/\delta').
\]
Let us remark that we have used the same functions $\Phi_{1, h}$ and $f_1$ in both $v_s^{(0)}$ and $v_s^{(1)}$, which is possible since both $\Phi_{1, h}$ and $f_1$ are  globally defined for all $r\in (-2\varepsilon, \frac{L}{\sqrt{1-\beta^2}}+2\varepsilon)$ and are independent of $y$.
Moreover, the phase functions $\psi^{(0)}$ and $\psi^{(1)}$ solve the same equations with identical  initial data in $\Omega_0$ and $\Omega_1$, and the local coordinates $ \kappa_0$ and $ \kappa_1$ agree on $\kappa_0^{-1}((I_0\cap I_1)\times B)$. Thus, it follows that
$\psi^{(1)}=\psi^{(0)}$ in $\Omega_0\cap \Omega_1$.
Hence, we see that
\[
v_s^{(0)}=v_s^{(1)} \quad \text{ in } \Omega_0\cap \Omega_1.
\]
Proceeding similarly, we obtain quasimodes $v_s^{(2)}, \dots, v_s^{(R+1)}$ such that
\begin{equation}
\label{eq:quasi_corres}
v_s^{(\ell)}(t, x_1, \cdot)=v_s^{(\ell+1)}(t, x_1, \cdot) \quad \text{in} \quad \Omega_\ell\cap \Omega_{\ell+1}
\end{equation}
for all $t$ and $x_1$. In the case where $\gamma$ self-intersects at $r=0$, we shall begin with the construction of $v^{(1)}$ by fixing initial conditions for the Riccati equation \eqref{eq:Riccati} at $r=0$ and find $v^{(0)}$ by going backwards.

Let $\chi_j(r)$ be a partition of unity subordinate to the intervals $(I_\ell)_{\ell=0}^{R+1}$, and let $\tilde \chi_\ell(t,x_1,r, y)=\chi_\ell(r)$. We define a smooth function
\[
v_s=\sum_{\ell=0}^{R+1} \tilde{\chi}_\ell v_s^{(\ell)} \quad \text{in $\R^2\times \hat M_0$.}
\]
Let $z_1, \dots, z_{R'}\in M_0$, $R'< R$,
be the distinct self-intersection points of $\gamma$ that correspond to the self-intersection times $0\le r_1<\cdots<r_R$. Let $V_j$, $j=1, \dots, R'$, be a small neighborhood of $z_j$ in $\hat{M}_0$. Arguing as in \cite[Proposition 7.5]{Kenig_Salo} and utilizing the definition of the intervals $I_0,\ldots, I_{R+1}$ together with \eqref{eq:quasi_corres}, we choose a finite cover $W_0,\ldots W_S$ of the remaining portion of the geodesic $\gamma$ such that each $W_k\subset \Omega_{\ell(k)}$ for some $\ell(k)\in \{0,\ldots,R+1\}$, $k=1,\ldots,S$. Consequently, we get the following open cover of the support:
\[
\supp(v_s(t,x_1, \cdot))\cap M_0 \subset \left(\bigcup_{j=1}^{R'}V_j\right) \cup \left(\bigcup_{k=0}^{S}W_{k}\right).
\]
Furthermore,  the restriction of the quasimode $v_s$ to  $V_j$ and $W_{k}$ can be written as
\begin{equation}
\label{eq:finite_sum_v}
v_s(t,x_1,\cdot)|_{V_j}=\sum_{\ell: \gamma(r_\ell)=z_j}v_s^{(\ell)}(t,x_1,\cdot)
\quad \text{ and } \quad
v_s(t,x_1, \cdot)|_{W_{k}}=v_s^{(\ell(k))}(t,x_1, \cdot),
\end{equation}
respectively. This representation follows from the construction of the intervals  $(I_\ell)_{\ell=0}^{R+1}$ and the associated partition of unity, as well as choosing the set $V_j$ sufficiently small.

It still remains to verify that the estimates in \eqref{eq:estimate_v} hold globally. To that end, in view of \eqref{eq:finite_sum_v}, the quasimode $v_s$ is a finite sum of $v^{(\ell)}$. Therefore, both estimates in \eqref{eq:estimate_v} follow directly from the corresponding local bounds \eqref{eq:estimate_v_int} and \eqref{eq:est_op_regu}. Moreover, the estimate
\[
\|v_s(t,x_1,\cdot)\|_{L^2(\p M_0)}=\cO(1)
\]
follows from the estimate \eqref{eq:est_v_bdyM} applied to each $v_s^{(\ell)}$.
The construction of the Gaussian beam quasimode $v_s$ is now complete.

We next construct a Gaussian beam solution for the  operator $e^{-\frac{1}{h^2} \beta^2 t - \frac{1}{h} x_1} h^2\mathcal{L}_{g, A, q} e^{\frac{1}{h^2} \beta^2 t + \frac{1}{h} x_1}$. As the argument is very similar to the previous construction, we omit the common steps and highlight only the differences.

Consider the ansatz $w_s$ of the form
\[
w_s(t,x_1, \tau, y;h) = e^{is\psi(\tau, y)} b(t,x_1,\tau, y;h),
\]
where $\psi$ is a phase function, and the amplitude $b$ depends implicitly on the semiclassical parameter $h$.
Using \eqref{eq:computation_time_derivative}--\eqref{eq:computation_first_order_2}, we obtain
\begin{equation}
\label{eq:conjugated_L2}
\begin{aligned}
&e^{-s^2 \beta^2 t - s x_1} \mathcal{L}_{g, A, q} e^{s^2 \beta^2 t + s x_1} \left( e^{is\psi(\tau, y)} b \right)
\\
=&
e^{is\psi} \left[ s^2 \left( \langle \nabla_{g_0} \psi, \nabla_{g_0} \psi \rangle_{g_0} - (1 - \beta^2) \right) b
\right.
\\
& \left. +
s \left( -2 \p_{x_1} b - 2i \langle \nabla_{g_0} \psi, \nabla_{g_0} b \rangle_{g_0} - i (\Delta_{g_0} \psi) b - 2A_1b -2i \langle A(x_1, \cdot), \nabla_{g_0} \psi \rangle_{g_0} b \right) + \mathcal{L}_{g,A, q} b \right].
\end{aligned}
\end{equation}
As the eikonal equation is the same as the one in \eqref{eq:conjugated_L1}, the phase function $\psi(\tau, y)$ is given by \eqref{eq:phase_form}.

We next seek an amplitude $b$ of the form
\begin{equation}
b(t, x_1, \tau, y, h; \zeta) = h^{-\frac{n-2}{4}} b_0(t, x_1, \tau; \zeta) \chi\left( \frac{y}{\delta'} \right),
\end{equation}
where $b_0 \in C^\infty(\mathbb{R}_t \times \mathbb{R}_{x_1} \times [\tau_0 - \delta, \tau_0 + \delta])$ is independent of $y$. Arguing similarly as in the construction of $a_0$, we see that $b_0$ satisfies the transport equation

\begin{align*}
&\left(\p_{x_1} + i \sqrt{1 - \beta^2} \p_\tau \right)b_0
\\
&= -\frac{1}{2}\left[i \sqrt{1 - \beta^2}\tr H(\tau) + 2 (A_\zeta)_1(t,x_1,\tau,0) + 2i  \sqrt{1 - \beta^2} (A_\zeta)_\tau (t,x_1,\tau,0)\right]b_0.
\end{align*}
Using the change of variable function $S(\tau) = \sqrt{1-\beta^2}\tau$ again, we get
\begin{equation}
\label{eq:transport_b0_new}
\begin{aligned}
&\left(\p_{x_1} + i \p_r\right) b'_0
\\
&= -\frac{1}{2} \left[i \sqrt{1 - \beta^2} \tr H(\sqrt{1 - \beta^2} r) + 2 (A_\zeta)_1 (t,x_1, r, 0) + 2i \sqrt{1 - \beta^2} (A_\zeta)_\tau (t, x_1, r, 0) \right] b'_0,
\end{aligned}
\end{equation}
where $b'_0 = b_0 \circ S$, $(A_\zeta)_1' = (A_\zeta)_1 \circ S$, and
$(A_\zeta)_\tau' = (A_\zeta)_\tau \circ S$.

We next write  $\overline{\p} = \frac {1}{2} \left(\p_{x_{1}} + i \p_r\right)$ and look for a solution to \eqref{eq:transport_b0_new} of the form
\begin{equation}
b'_0 (t, x_1, r; \zeta) = e^{\Phi_{2,\zeta} (t, x_1, r) + f_2(r)}\eta(t,x_1,r)
\end{equation}
such that $\overline{\p}\eta=0$. Then $\Phi_{2,\zeta}$ and $f_2$ must satisfy the following equations, respectively:

\begin{align*}
\overline{\p}{\Phi_{2,\zeta}} (t, x_1, r) = -\frac{1}{2} \left[ (A_\zeta)_1 (t,x_1, r, 0) + i \sqrt{1 - \beta^2} (A_\zeta)_\tau (t, x_1, r, 0) \right],
\end{align*}
and
\begin{align*}
\p_{r} f_2(r) = -\frac{1}{2} \sqrt{1 - \beta^2} \tr H(\sqrt{1 - \beta^2} r).
\end{align*}
From here, we utilize the same arguments as in the construction of $v_s$ to obtain a Gaussian beam quasimode $w_s\in C^\infty(Q)$ satisfying the estimates in \eqref{eq:estimate_w}. This completes the proof of Theorem \ref{prop:Gaussian_beam}.
\end{proof}

\subsection{Concentration property of   Gaussian beam quasimodes}
\label{Ssec:concetration}

We have established that, for each non-tangential geodesic $ \gamma\colon[0, \frac{L}{\sqrt{1-\beta^2}}]\to M_0$ and  $0<h \ll 1$, there exist smooth functions $\Phi_{1,h},\Phi_{2,h}$ in $[0,T] \times \overline{J}_{x_1} \times [0, \frac{L}{\sqrt{1-\beta^2}}]$ satisfying the equations
\[
(\p_{x_1}-i\p_r) \Phi_{1,h}(t, x_1, r) = (\overline A_{h})_1' (t,x_1, r, 0) -i  \sqrt{1 - \beta^2} (\overline A_{h})'_\tau (t, x_1, r, 0) ,
\]
and
\[
(\p_{x_1}+i\p_r) \Phi_{2,h}(t, x_1, r)=-(A_h)'_1 (t,x_1, r, 0) - i \sqrt{1 - \beta^2} (A_h)'_\tau (t, x_1, r, 0),
\]
respectively, where $A_h'=A_h \circ S$, and the change of coordinates function $S(\tau) = \sqrt{1-\beta^2}\tau$.

By arguing similarly to the proof of \cite[Lemma 4.3]{Liu_Saksala_Yan}, we obtain the following result describing the behavior of these functions as $\zeta \to 0$.
\begin{lem}
\label{lem:Phi_nonregularized}
Let $\beta \in (\frac{1}{\sqrt{3}},1)$, and let $\gamma\colon[0, \frac{L}{\sqrt{1-\beta^2}}]\to M_0$ be a non-tangential geodesic in $(M_0,g_0)$ as in Theorem \ref{prop:Gaussian_beam}. Then there exist functions $\Phi_1$, $\Phi_2\in  C([0,T] \times \overline{J}_{x_1} \times [0, \frac{L}{\sqrt{1-\beta^2}}])$ that satisfy the equations

\[
(\p_{x_1}-i\p_r) \Phi_{1}(t, x_1, r) = \overline A_1' (t,x_1, r, 0) -i  \sqrt{1 - \beta^2} \overline A'_\tau (t, x_1, r, 0) ,
\]
and
\[
(\p_{x_1}+i\p_r) \Phi_{2}(t, x_1, r)=-A'_1 (t,x_1, r, 0) - i \sqrt{1 - \beta^2} A'_\tau (t, x_1, r, 0),
\]
respectively. Furthermore, the function $\Phi_j$, $j=1,2$, satisfies the following estimate:
\begin{equation}
\label{eq:bound_diff_phi}
\|\Phi_{j, h}-\Phi_j\|_{L^\infty([0,{T}] \times \overline{J}_{x_1} \times [0, \frac{L}{\sqrt{1-\beta^2}}])}=o(1), \quad \zeta \to 0.
\end{equation}
\end{lem}

In the following proposition, we show that the  Gaussian beam quasimodes obtained in Theorem \ref{prop:Gaussian_beam} concentrate along the geodesic when $h \to 0$.

\begin{prop}
\label{prop:limit_behavior}
Let $(M,g)$ be a CTA manifold. Let $0<h\ll1$ be a small parameter, $\beta \in (\frac{1}{\sqrt{3}},1)$, and $s=\frac{1}{h}$. Suppose that  $\gamma\colon[0, \frac{L}{\sqrt{1-\beta^2}}]\to M_0$ is a non-tangential geodesic in the transversal manifold $(M_0,g_0)$. Let $J_{x_1}$ be an open and bounded interval such that the  $x_1$-coordinate of each point in $Q$ lies in $J_{x_1}$.
Let $v_s$ and $w_s$ be the Gaussian beam quasimodes from Theorem \ref{prop:Gaussian_beam}. Then for each function $\phi \in C(M_0)$ and
$(t,x_1)\in  [0,{T}]\times \overline{J}_{x_1}$,
we have
\begin{equation}
\label{eq:limit_prod_vw}
\begin{aligned}
&\lim_{h \to 0} \int_{M_0} \overline{v}_s (t,x_1, \cdot)w_s(t,x_1, \cdot) \phi   dV_{g_0}
=
(1-\beta^2)^{-\frac{n-6}{4}}\int_{0}^{\frac{L}{\sqrt{1-\beta^2}}} e^{\overline{\Phi_1}(t,x_1, r)+\Phi_2(t,x_1, r)}
\eta(t,x_1,r)\phi(\gamma(r))dr.
\end{aligned}
\end{equation}
Furthermore, for any one-form $A\in C(M, T^\ast M)$, we have
\begin{equation}
\label{eq:concentration_one_form_ws}
\begin{aligned}
&\lim_{h \to 0} h \int_{M_0} \n{A, dw_s(t,x_1, \cdot) }_g\overline{v}_s (t,x_1, \cdot)\phi   dV_{g_0}
\\
&=(1-\beta^2)^{-\frac{n-8}{4}}\int_{0}^{\frac{L}{\sqrt{1-\beta^2}}}  iA(\dot \gamma(r)) e^{\overline{\Phi_1}(t,x_1, r)+\Phi_2(t,x_1, r)}
\eta(t,x_1,r)\phi(\gamma(r))dr.
\end{aligned}
\end{equation}
Here the functions $\Phi_1, \Phi_2\in C( [0,{T}] \times \overline{J}_{x_1} \times [0, \frac{L}{\sqrt{1-\beta^2}}])$ are as in Lemma \ref{lem:Phi_nonregularized},

and $\eta \in C^\infty( [0,T] \times \overline{J}_{x_1} \times [0, \frac{L}{\sqrt{1-\beta^2}}])$ with $(\p_{x_1}+i\p_r)\eta=0$.
\end{prop}

\begin{proof}
The derivation of \eqref{eq:limit_prod_vw} follows from the same arguments as in the proof of \cite[Theorem 4.4]{Liu_Saksala_Yan}. Therefore, we shall only present the proof of \eqref{eq:concentration_one_form_ws}.
To this end, by a standard partition of unity argument, it suffices to consider the cases where $\phi\in C_0(V_j\cap M_0)$ and $\phi\in C_0(W_k\cap M_0)$, where the sets $V_j$ and $W_k$ are the same as in Step 5 in the proof of Theorem \ref{prop:Gaussian_beam}.

We first consider the case that $\phi\in C_0(W_k\cap M_0)$ for some $k$. Here $\supp \phi$ may extend to $\p M_0$, and we extend $\phi$ by zero outside of $W_k\cap M_0$.
Let us recall from  Theorem \ref{prop:Gaussian_beam}  that we have constructed the following Gaussian beam quasimodes on $\supp \phi$ :
\begin{equation}
\label{eq:quasimode_supp_psi}
\begin{aligned}
&v_s(t, x_1, r, y)=e^{is\psi(\sqrt{1-\beta^2}r, y)}h^{-\frac{n-2}{4}}e^{\Phi_{1,h}(t, x_1, r)+f_1(r)}\chi(y/\delta'),
\\
&w_s(t, x_1, r, y)=e^{is\psi(\sqrt{1-\beta^2}r, y)}h^{-\frac{n-2}{4}} e^{\Phi_{2,h}(t, x_1, r)+f_2(r)}\eta(t,x_1, r)\chi(y/\delta').
\end{aligned}
\end{equation}

Let us also recall that we have denoted  $a_0(t,x_1, r):=   e^{\Phi_{1,h}(t, x_1, r)+f_1(r)}\eta(t,x_1, r)$ in the proof of Theorem \ref{prop:Gaussian_beam}. Using the Fermi coordinates $z = (\tau, y)$ and the fact that $g=e\oplus g_0$, we compute that
\begin{align*}
&h \int_{M_0} \langle A, d w_s(t,x_1, \cdot) \rangle_g \overline{v}_s (t,x_1, \cdot)    \phi   dV_{g_0}
\\
= &
h \int_{M_0}  A_1 e^{is\psi}   h^{-\frac{n-2}{4}}   \chi\left(\frac{y}{\delta'}\right)   \partial_{x_1} a_0 (t,x_1, \cdot)    \overline{v}_s(t,x_1, \cdot)    \phi   dV_{g_0}
\\
&+
i\int_{M_0}  g_0^{kj} A_k  (\partial_{z_j} \psi)   w_s(t,x_1, \cdot)   \overline{v}_s(t,x_1, \cdot)    \phi   dV_{g_0}
\\
&+
h \int_{M_0}  g_0^{kj} A_k e^{is\psi}   h^{-\frac{n-2}{4}}   (\partial_{z_j} a_0(t,x_1, \cdot)  )   \chi\left(\frac{y}{\delta'}\right)   \overline{v}_s (t,x_1, \cdot)    \phi   dV_{g_0}
\\
&+
h \int_{M_0} g_0^{kj} A_k e^{is\psi}   h^{-\frac{n-2}{4}}   a_0 (t,x_1, \cdot)    \left(\partial_{z_j} \chi\left(\frac{y}{\delta'}\right)\right)   \overline{v}_s(t,x_1, \cdot)     \phi   dV_{g_0},
\end{align*}
where $j,k=2,\dots, n$.

We next show that the first, third, and fourth integrals on the right-hand side of the expression above vanish  as $h \to 0$. For the third integral, as $0<\alpha < 1/2$, by the Cauchy-Schwarz inequality, the estimates \eqref{eq:estimate_v}, \eqref{eq:bound_deri_b0_h_old}, and \eqref{eq:L2_power_y} with $k=0$, we have
\begin{align*}
h\left| \int_{{M_0}} g_0^{kj} A_k e^{is\psi} h^{-\frac{n-2}{4}}  (\partial_{z_j} a_0) \chi\left(\frac{y}{\delta'}\right)\overline{v}_s  \phi    dV_{g_0} \right|
&\leq \mathcal{O}(h^{1-\alpha}) \left\| e^{is\psi} h^{-\frac{n-2}{4}} \right\|_{L^2(\{|y| \leq \delta'/2 \})}\left\| v_s  \right\|_{L^2(M_0)}\\&= \mathcal{O}(h^{1-\alpha}), \quad h\to 0.
\end{align*}
The first integral is estimated similarly using the bound for $\partial_{x_1}a_0$ given in \eqref{eq:bound_deri_b0_h_old}. For the fourth integral, the derivative of the cutoff function $\chi$ is supported away from the geodesic. Therefore, this term is of order  $\cO(h^\infty)$.

Turning our attention to the second integral, in view of \eqref{eq:phase_form} and \eqref{eq:quasimode_supp_psi}, and using the identities
\[
|g_0(r,y)|^{1/2} = \sqrt{1-\beta^2} + \cO(|y|^2), \quad \p_\tau \psi(r, y) = \sqrt{1-\beta^2} +\cO(|y|^2), \quad \text{ and } \quad \p_y \psi(r,y) = \cO(|y|),
\]
we compute that
\begin{align*}
&i \int_{M_0} g_0^{kj} A_k \left( \partial_{z_j} \psi \right) w_s(t,x_1, \cdot) \overline{v}_s(t,x_1, \cdot) \phi    dV_{g_0}
\\
& =  i\sqrt{1-\beta^2} \int_{0}^{\frac{L}{\sqrt{1-\beta^2}}}\int_{\R^{n-2}} g_0^{kj} A_k(t,x_1, \sqrt{1-\beta^2}r, y) \left( \partial_{z_j} \psi \right)e^{-\frac{2}{h} \Im \psi} h^{-\frac{n-2}{2}}
\\
& \qquad  \quad  \qquad \qquad \qquad\quad \times e^{\overline{\Phi_{1, h}}(t,x_1, r)+\Phi_{2,h}(t,x_1, r)+\overline{f_1}(r)+f_2(r)} \chi^2(y/\delta')\eta(t, x_1, r)\phi(r, y)|g_0|^{1/2}dydr
\\
&\ = i\sqrt{1-\beta^2} \int_{0}^{\frac{L}{\sqrt{1-\beta^2}}}\int_{\R^{n-2}}g_0^{kj} \left(\sqrt{1-\beta^2}A_\tau(t,x_1, \sqrt{1-\beta^2}r, y)+\mathcal{O}(|y|)\right)
\\
& \quad  \quad \quad \quad \quad \quad \quad \quad \quad \quad  \times  e^{-\frac{1}{h} \sqrt{1-\beta^2}\Im H(\sqrt{1-\beta^2}r)y\cdot y}
h^{-\frac{n-2}{2}}\chi^2(y/\delta')\eta(t,x_1, r)
\\
& \quad  \quad \quad \quad \quad \quad \quad \quad \quad \quad \times  e^{\overline{\Phi_{1, h}}(t,x_1, r)+\Phi_{2,h}(t,x_1, r)+\overline{f_1}(r)+f_2(r)}\phi(r,y) \left(\sqrt{1-\beta^2}+\mathcal{O}(|y|^2)\right)dydr.
\end{align*}
Performing the change of variables $y\mapsto h^{\frac{1}{2}}y$, we obtain
\begin{equation}
\label{eq:lim_prod}
\begin{aligned}
&i \int_{M_0} g_0^{kj} A_k \left( \partial_{z_j} \psi \right) w_s(t,x_1, \cdot) \overline{v}_s(t,x_1, \cdot) \phi    dV_{g_0}
\\
& = i (1-\beta^2)\int_{0}^{\frac{L}{\sqrt{1-\beta^2}}}\int_{\R^{n-2}}
g_0^{kj} \left(\sqrt{1-\beta^2}A_\tau(t,x_1, \sqrt{1-\beta^2}, h^{\frac{1}{2}}y)+\mathcal{O}(h^{\frac{1}{2}}|y|)\right)
\\
&\qquad \qquad \qquad \qquad \quad \quad \times e^{- \sqrt{1-\beta^2}\Im H(\sqrt{1-\beta^2}r)y\cdot y}  \chi^2(h^{1/2}y/\delta')
e^{\overline{\Phi_{1, h}}(t,x_1, r)+\Phi_{2,h}(t,x_1, r)+\overline{f_1}(r)+f_2(r)}
\\
&\qquad \qquad \qquad \qquad \quad \quad \times \phi(r,h^{\frac{1}{2}}y) (1+h\mathcal{O}(|y|^2)) \eta(t,x_1, r)   dydr.
\end{aligned}
\end{equation}

Let us next analyze the expressions on the right-hand side of \eqref{eq:lim_prod}.  Due to continuity, we get the following pointwise limits as $h\to 0$:
\[
\chi^2(h^{\frac{1}{2}}y/\delta')\to 1, \quad \phi(r, h^{\frac{1}{2}}y)\to \phi(r, 0)=\phi(\gamma(r)), \quad \Phi_{j,h} \to \Phi_j, \: j=1,2.
\]
In particular, the last limit follows from \eqref{eq:bound_diff_phi}. Therefore, by the dominated convergence theorem, we have
\begin{equation}
\label{eq:lim_prod_after_limit}
\begin{aligned}
&i \int_{M_0} g_0^{kj} A_k \left( \partial_{z_j} \psi \right) w_s(t,x_1, \cdot) \overline{v}_s(t,x_1, \cdot) \phi    dV_{g_0}
\\
& = i (1-\beta^2)^{3/2} \int_{0}^{\frac{L}{\sqrt{1-\beta^2}}}
g_0^{kj}  A_\tau(t,x_1, \sqrt{1-\beta^2}r,0) \left(\int_{\R^{n-2}} e^{- \sqrt{1-\beta^2}\Im H(\sqrt{1-\beta^2}r)y\cdot y} dy \right)
\\
&\qquad \qquad \qquad \qquad \quad \quad \times
e^{\overline{\Phi_{1, h}}(t,x_1, r)+\Phi_{2,h}(t,x_1, r)+\overline{f_1}(r)+f_2(r)} \phi(\gamma(r))   \eta(t,x_1, r) dr.
\end{aligned}
\end{equation}

To further analyze the integral with respect to $y$,  let us recall that
\[
\sqrt{1-\beta^2} \Im H(\sqrt{1-\beta^2}r)y\cdot y\ge d|y|^2 \quad \text{ and } \quad \int_{\R^{n-2}}e^{-d|y|^2}dy<\infty.
\]
Thus, by performing a change of variables $y\mapsto (1-\beta^2)^{-\frac{1}{4}}y$, we obtain
\begin{align*}
\int_{\R^{n-2}} e^{-\sqrt{1-\beta^2}\Im H(\sqrt{1-\beta^2}r)y\cdot y}dy&= \int_{\R^{n-2}} (1-\beta^2)^{-\frac{n-2}{4}}e^{-\Im H( \sqrt{1-\beta^2}r)y\cdot y} dy
\\
&= \frac{\pi^{\frac{n-2}{2}}(1-\beta^2)^{-\frac{n-2}{4}}}{\sqrt{\det(\Im H( \sqrt{1-\beta^2}r))}}.
\end{align*}
Let $\tau_0=\sqrt{1-\beta^2}r_0$. Then it follows from  \cite[Lemma 2.58]{KKL_book} that
\[
\det(\Im H( \sqrt{1-\beta^2}r)) = \det(\Im H( \sqrt{1-\beta^2}r_0)) e^{-2 \int_{r_0}^{r} \sqrt{1-\beta^2}\mathrm{tr} \Re H( \sqrt{1-\beta^2}w)dw}.
\]
Thus, it holds that
\begin{equation}
\label{eq:Gaussian_int}
\int_{\R^{n-2}} e^{-\sqrt{1-\beta^2}\Im H(\sqrt{1-\beta^2}r)y\cdot y}dy=\frac{\pi^{\frac{n-2}{2}}(1-\beta^2)^{-\frac{n-2}{4}}e^{\int_{r_0}^{r} \sqrt{1-\beta^2}\mathrm{tr} \Re(H( \sqrt{1-\beta^2}w))dw}}{\sqrt{\det(\Im H(\sqrt{1-\beta^2}r_0))}}.
\end{equation}

Turning our attention to the integral with respect to $r$ in \eqref{eq:lim_prod_after_limit},  let us recall from the proof of Theorem \ref{prop:Gaussian_beam}  that the function $f_j$, $j=1,2$, satisfies the equation
\[
\p_r f_j(r)= -\frac{1}{2} \sqrt{1-\beta^2}\tr H(\sqrt{1-\beta^2}r).
\]
Hence, we immediately get
\[
\p_r (\overline{f_1}(r)+f_2(r)) =- \sqrt{1-\beta^2} \tr \Re H(\sqrt{1-\beta^2}r).
\]
By the fundamental theorem of calculus, we have
\begin{equation}
\label{eq:f1_plus_f2}
\overline{f_1}(r)+f_2(r)=\overline{f_1}(r_0)+f_2(r_0)-\int_{r_0}^{r} \sqrt{1-\beta^2} \mathrm{tr} \Re(H(\sqrt{1-\beta^2}w))dw.
\end{equation}
We now choose $f_1(r_0)$ and $f_2(r_0)$ such that
\begin{equation}
\label{eq:choice_r0}
\frac{e^{\overline{f_1}(r_0)+f_2(r_0)}\pi^{\frac{n-2}{2}}}{\sqrt{\det(\Im H(\sqrt{1-\beta^2}r_0))}}=1.
\end{equation}
Hence, by combining  \eqref{eq:Gaussian_int}--\eqref{eq:choice_r0}, we obtain
\begin{equation}
\label{eq:prod_exp_int}
e^{\overline{f_1}(r)+f_2(r)}\left(\int_{\R^{n-2}} e^{-\Im H(\sqrt{1-\beta^2}r)y\cdot y}dy\right)=(1-\beta^2)^{-\frac{n-2}{4}}.
\end{equation}

Therefore, we get  \eqref{eq:concentration_one_form_ws} when $\supp{\phi}\subset W_k$ by substituting the above identity into \eqref{eq:lim_prod_after_limit}.

We now assume that $\mathrm{supp}(\phi) \subset V_j$. Let $z_j$ be a self-intersection point of the geodesic $\gamma$ as in Step 5 of the proof of Theorem \ref{prop:Gaussian_beam}. In this case, we write the Gaussian beam quasimodes on $\mathrm{supp}(\phi)$ as follows:
\[
v_{s} = \sum_{\gamma(r_{l}) = z_j} v_{s}^{(l)}, \quad w_{s} = \sum_{\gamma(r_{l}) = z_j} w_{s}^{(l)}.
\]
Then we have
\[
\begin{aligned}
h \int_{\ M_{0}} \langle A, dw_{s} \rangle_{g} \overline{v_{s}}   \phi    dV_{g_{0}} &=h\sum_{l : \gamma(r_{l}) = z_j} \int_{\ M_{0}} \langle A, dw_{s}^{(l)} \rangle_{g} \overline{v_{s}^{(l)}}   \phi    dV_{g_{0}} \\
&\quad + h  \sum_{l \neq l^{\prime},  \gamma(r_{l}) = \gamma(r_{l^{\prime}}) = z_j} \int_{\ M_{0}} \langle A, dw_{s}^{(l)} \rangle_{g} \overline{v_{s}^{(l^{\prime})}}   \phi    dV_{g_{0}}.
\end{aligned}
\]
Our goal is to show that
\[
\lim_{h \to 0} h \int_{\ M_{0}} \langle A, dw_{s}^{(l)} \rangle_{g} \overline{v_{s}^{(l^{\prime})}}   \phi    dV_{g_{0}} = 0, \quad l \neq l'.
\]
Then the limit \eqref{eq:concentration_one_form_ws} follows from the first part of the proof.

Thanks to the computations from the previous case, we only have to verify that
\[
\lim_{h \to 0} \int_{\ M_{0}} g_{0}^{kj}   A_{k}   (\partial_{z_{j}} \psi^{(l)})   w_{s}^{(l)}   \overline{v_{s}^{(l^{\prime})}}   \phi    dV_{g_{0}} = 0.
\]

Arguing as in the proof of  \cite[Proposition 3.1]{Ferreira_Kur_Las_Salo}, let us write
\[
v_s^{(l')}=e^{\frac{i}{h} \Re \psi^{(l')}}\alpha^{(l')}, \quad \alpha^{(l')}=e^{-\frac{1}{h}\Im \psi^{(l')}}a^{(l')},
\]
and
\[
w_s^{(l)}=e^{\frac{i}{h}  \Re \psi^{(l)}}\omega^{(l)}, \quad \omega^{(l)}=e^{-\frac{1}{h}\Im \psi^{(l)}}b^{(l)},
\]
which implies that
\begin{equation}
\label{eq:prod_vw_rewrite}
\overline{v_s^{(l')}}w_s^{(l)}= e^{\frac{i}{h} \sigma} \overline{\alpha^{(l')}}\omega^{(l)}, \quad \sigma= \Re \psi^{(l)} - \Re \psi^{(l')}.
\end{equation}
Hence, it suffices to show that
\begin{equation}
\label{eq:prod_vanish_rewrite}
\lim_{h \to 0}\int_{M_0} g_{0}^{kj} A_{k}   (\partial_{z_{j}} \psi^{(l)})e^{\frac{i}{h} \sigma} \omega^{(l)}(t,x_1, \cdot)\overline{\alpha^{(l')}}(t,x_1, \cdot)\phi   dV_{g_0}=0, \quad l\ne l'.
\end{equation}
To this end, let us write $\phi=\phi_h+(\phi-\phi_h)$, where  $\phi_h \in C^\infty_0(V_j\cap M_0)$, but its support can meet $\p M_0$, and $\phi-\phi_h$ is continuous. Due to the estimates \eqref{eq:estimate_v} and \eqref{eq:estimate_w}, it holds that
\[
\|\alpha^{(l')}\|_{L^2(V_j\cap M_0)},\|\omega^{(l)}\|_{L^2(V_j\cap M_0)}=\cO(1).
\]
Thus, we apply H\"older's inequality and the estimate \eqref{eq:est_diff_regu} to get
\begin{equation}
\label{eq:est_difference}
\begin{aligned}
\left|\int_{M_0}g_{0}^{kj} A_{k} (\partial_{z_{j}} \psi^{(l)})e^{\frac{i}{h}\sigma}\omega^{(l)}(t,x_1, \cdot)\overline{\alpha^{(l')}}(t,x_1, \cdot)(\phi-\phi_h)   dV_{g_0}\right|
=o(1), \quad h \to 0,
\end{aligned}
\end{equation}
where the $L^p$-norms are taken over the set $V_j\cap M_0$.

To analyze the term involving  $\phi_h$,  by \eqref{eq:g0_near_geo} and \eqref{eq:phase_form}, the gradients of $\Re\psi^{(l)}$ and $\Re\psi^{(l')}$ at $z_j$ are parallel to $\dot{\gamma}(r_l)$ and $\dot \gamma(r_{l'})$, respectively. Since the geodesic $\gamma$ intersects itself transversally at $z_j$, we have $\nabla_{g_0}\sigma(z_j)\ne 0$. Hence, by shrinking $V_j$ if necessary, we may assume that $\sigma$ has no critical points in $V_j$.

This implies that  the vector field
\[
L=\frac{g_{0}^{kj} A_{k} (\partial_{z_{j}} \psi^{(l)})\overline{\alpha^{(l')}}\omega^{(l)}\phi_h}{|\nabla_{g_0}\sigma|^2}\langle \nabla_{g_0}\sigma, \nabla_{g_0} \cdot\rangle_{g_0}
\]
is well-defined and satisfies the property
\[
e^{\frac{i}{h}\sigma}g_{0}^{kj} A_{k} (\partial_{z_{j}} \psi^{(l)})\overline{\alpha^{(l')}}\omega^{(l)}\phi_h
=
-{i}{h}L(e^{\frac{i}{h}\sigma}).
\]

Then we integrate by parts to deduce from  \eqref{eq:prod_vanish_rewrite}  that
\begin{equation}
\label{eq:prod_psi1}
\begin{aligned}
&\int_{M_0}g_{0}^{kj} A_{k} (\partial_{z_{j}} \psi^{(l)})e^{{\frac{i}{h}}\sigma}\overline{\alpha^{(l')}}(t,x_1, \cdot)\omega^{(l)}(t,x_1, \cdot)\phi_h   dV_{g_0}
\\& \ =
-{i}{h}\int_{V_j\cap \p M_0} \frac{\p_\nu \sigma}{ |\nabla_{g_0}\sigma|^2} e^{\frac{i}{h}\sigma} g_{0}^{kj} A_{k} (\partial_{z_{j}} \psi^{(l)})\overline{\alpha^{(l')}}(t,x_1, \cdot)\omega^{(l)}(t,x_1, \cdot)\phi_h   dS_{g_0}
\\& \ \ \ \ +{i}{h} \int_{M_0} e^{\frac{i}{h}\sigma}
\div\left(\frac{g_{0}^{kj} A_{k} (\partial_{z_{j}} \psi^{(l)})\overline{\alpha^{(l')}}\omega^{(l)}\phi_h}{|\nabla_{g_0}\sigma|^2} \nabla_{g_0}\sigma\right)(t,x_1, \cdot)   dV_{g_0}.
\end{aligned}
\end{equation}

We next show that the boundary term on the right-hand side of \eqref{eq:prod_psi1} vanishes as $h \to 0$. To this end, using the  estimate \eqref{eq:est_v_bdyM}, we  get $\|\alpha^{(l')}\|_{L^2(\p M_0)}$, $\|\omega^{(l)}\|_{L^2(\p M_0)}=\cO(1)$. Furthermore, as $\sigma$ is real-valued and independent of $h$,
by the estimate \eqref{eq:est_diff_regu} and H\"older's inequality, we have
\begin{equation}
\label{eq:est_bdy_term}
-{i}{h}\int_{V_j\cap \p M_0} \frac{\p_\nu \sigma}{ |\nabla_{g_0}\sigma|^2} e^{\frac{i}{h}\sigma} g_{0}^{kj} A_{k} (\partial_{z_{j}} \psi^{(l)})\overline{\alpha^{(l')}}(t,x_1, \cdot)\omega^{(l)}(t,x_1, \cdot)\phi_h dS_{g_0}
=
\cO(h), \quad h \to 0.
\end{equation}

Turning our attention to the second term on the right-hand side of \eqref{eq:prod_psi1}, we compute that
\begin{align*}
\div \left(g_{0}^{kj} A_{k} (\partial_{z_{j}} \psi^{(l)})\overline{\alpha^{(l')}}\omega^{(l)}\phi_h \frac{\nabla_{g_0}\sigma}{|\nabla_{g_0}\sigma|^2}\right)
=& \left\langle \nabla_{g_0}(g_{0}^{kj} A_{k} (\partial_{z_{j}}\psi^{(l)})\overline{\alpha^{(l')}}\omega^{(l)}\phi_h), \frac{\nabla_{g_0}\sigma}{|\nabla_{g_0}\sigma|^2}\right\rangle_{g_0}
\\& \  +
g_{0}^{kj} A_{k} (\partial_{z_{j}} \psi^{(l)})\overline{\alpha^{(l')}}\omega^{(l)}\phi_h \div \left(\frac{\nabla_{g_0}\sigma}{|\nabla_{g_0}\sigma|^2}\right).
\end{align*}
Therefore, by the  estimates \eqref{eq:est_diff_regu}, $\|\alpha^{(l')}\|_{L^2(M_0)},\|\omega^{(l)}\|_{L^2(M_0)}=\cO(1)$, and H\"older's inequality, we get
\begin{equation}
\label{eq:est_div_1}
{i}{h} \int_{M_0} g_{0}^{kj} A_{k} (\partial_{z_{j}} \psi^{(l)})e^{\frac{i}{h}\sigma} \overline{\alpha^{(l)}}(t,x_1, \cdot)\omega^{(l')}(t,x_1, \cdot)\phi_h \div \left(\frac{\nabla_{g_0}\sigma}{|\nabla_{g_0}\sigma|^2}\right) dV_{g_0}
=\cO(h), \quad h \to 0.
\end{equation}

Let us now write
\begin{align*}
g_{0}^{kj} A_{k} (\partial_{z_{j}} \psi^{(l)})\overline{\alpha^{(l')}}\omega^{(l)}\phi_h
&=[g_{0}^{kj} A_{k} (\partial_{z_{j}} \psi^{(l)})][e^{-\frac{1}{h}(\Im\psi^{(l)}+\Im \psi^{(l')})}h^{-\frac{n-2}{2}}][a_0^{(l)}\overline{b_0^{(l')}}\chi^2(y/\delta')]\phi_h
\\
&=f_1f_2f_3\phi_h,
\end{align*}

Using the Cauchy-Schwarz inequality, together with the estimates \eqref{eq:est_diff_regu}, \eqref{eq:bound_deri_b0_h_old} and \eqref{eq:L2_power_y} with $k=0$, we obtain
\begin{equation}
\label{eq:est_div_2}
{h}\int_{M_0}
|e^{\frac{i}{h}\sigma} f_2
f_3\phi_h \langle \nabla_{g_0}f_1, \frac{\nabla_{g_0}\sigma}{|\nabla_{g_0}\sigma|^2}\rangle|
dV_{g_0}
=
\cO(h^{1-\alpha}), \quad  h \to 0,
\end{equation}
and
\begin{equation}
\label{eq:est_div_3}
{h}\int_{M_0}
|e^{\frac{i}{h}\sigma} f_1
f_2 \langle \nabla_{g_0}(f_3\phi_h), \frac{\nabla_{g_0}\sigma}{|\nabla_{g_0}\sigma|^2}\rangle|
dV_{g_0} = 
\cO(h^{1-\alpha}), \quad  h \to 0.
\end{equation}
When $\nabla_{g_0}$ acts on $f_2$,   by the Cauchy-Schwarz inequality, as well as estimates \eqref{eq:est_diff_regu} and \eqref{eq:L2_power_y} with $k=\frac{1}{2}$, we have

\begin{equation}
\label{eq:est_div_4}
\begin{aligned}
&{h}\int_{M_0}
|e^{\frac{i}{h}\sigma} f_1
f_3\phi_h \langle \nabla_{g_0}f_2, \frac{\nabla_{g_0}\sigma}{|\nabla_{g_0}\sigma|^2}\rangle|
dV_{g_0}
\\
&\lesssim
\|\phi_h\|_{L^\infty}\|\alpha^{(l')}(t,x_1, \cdot)\|_{L^2}\|\omega^{(l)}(t,x_1, \cdot)\|_{L^2}\int_{M_0}
h^{-\frac{n-2}{2}}|y|e^{-\frac{d}{h} |y|^2}
dV_{g_0}
=
\cO(h^{1/2}), \quad  h \to 0.
\end{aligned}
\end{equation}
Hence, by \eqref{eq:est_div_1}--\eqref{eq:est_div_4}, the second term on the right-hand side of \eqref{eq:prod_psi1} is of order $\cO(h^{1/2})$. Furthermore, combining it with the estimate \eqref{eq:est_bdy_term} gives us
\begin{equation}
\label{eq:est_cross}
\int_{M_0}g_{0}^{kj} A_{k} (\partial_{z_{j}} \psi^{(l)})e^{{\frac{i}{h}}\sigma}\overline{\alpha^{(l')}}(t,x_1, \cdot)\omega^{(l)}(t,x_1, \cdot)\phi_h   dV_{g_0} = \cO(h^{1/2}), \quad h\to 0.
\end{equation}
Therefore, the limit \eqref{eq:prod_vanish_rewrite} follows immediately from \eqref{eq:est_difference} and \eqref{eq:est_cross}.

This completes the proof of \eqref{eq:concentration_one_form_ws} when  $\mathrm{supp}(\phi) \subset V_j$.

The proof of Theorem \ref{prop:limit_behavior} is now complete.
\end{proof}

\subsection{Construction of the remainder term}
\label{subsec:remainder}

We conclude this section with the construction of the remainder term. Let $(M,g)$ be a CTA manifold as in Definition \ref{def:CTA_manifolds}, and let $Q=(0,T)\times M$ for any $T>0$.  For any point $(t,x)\in Q$, let us express its local coordinates  as $(t,x)=(t,x_1, x')$.  We no longer assume that the conformal factor $c=1$ in this subsection. By \eqref{eq:equivalence_operator}, we get
\[
c^{\frac{n+2}{4}}\circ \mathcal{L}_{c,g,A,q} \circ \conf = \mathcal{L}_{\tilde{g}, A, \tilde{q}},
\]
where $\tilde g = e\oplus g_0$ and $\tilde{q}=c(q-c^{\frac{n-2}{4}}\Delta_g(\conf))$. Thus, the function $u=\conf \tilde u$ is a solution to the equation $\mathcal{L}_{c,g,A,q} u=0$ in $Q$ if $\tilde u$ satisfies $\mathcal{L}_{\tilde g,A,\tilde q} \tilde u=0$ in $Q$.

Let us  recall that  $0<h \ll 1$. In view of the discussion above, it suffices to construct a CGO solution to the equation
\begin{equation}
\label{eq:equation_u_tilde}
\mathcal{L}_{\tilde g,A,\tilde q} u=0 \quad \text{ in } Q
\end{equation}
of the form
\[
\tilde u (t,x) = e^{-s^2 \beta^2 t - s x_1}(v_s+r_1).
\]
Here $v_s$ is the Gaussian beam quasimode given in Theorem \ref{prop:Gaussian_beam}, and $r=r_s$ is a remainder term vanishing at a desirable rate as $h \to 0$. We observe that $\tilde u$ solves \eqref{eq:equation_u_tilde} if $r$ satisfies the following equation:
\begin{equation}
\label{eq:equation_remainder}
e^{s^2 \beta^2 t +  s x_1}h^2\mathcal{L}_{\tilde{g}, \tilde{a}, \tilde{q}}e^{-s^2 \beta^2 t -  s x_1}r=-e^{s^2 \beta^2 t +  s x_1} h^2 \mathcal{L}_{\tilde{g}, \tilde{a}, \tilde{q}}e^{-s^2 \beta^2 t -  s x_1}v_s.
\end{equation}
Thus, an application of Lemma \ref{lem:solvability} and the estimate \eqref{eq:estimate_v} yields that there exists a function $r\in H^1(0,T;H^{-1}(M))\cap L^2(0,T;H^1(M))$ satisfying \eqref{eq:equation_remainder} such that $\|r\|_{L^2(0,T;H^1_\scl(M))}=o(h)$ as $h \to 0$.

To summarize our discussion above, we obtain the following result.

\begin{thm}
\label{thm:existence_remainder}
Let $(M,g)$ be a CTA manifold, and let $Q=(0,T)\times M$ for any $T>0$. Let $A\in W^{1,\infty}(Q, T^*Q)$ and $q\in C(\overline Q)$. Then for each sufficiently small semiclassical parameter $h$, the equation $\mathcal{L}_{c,g,A,q}^\ast u_1 =0$ in $Q$  admits a solution $u_1 \in H^1(0,T;H^{-1}(M))\cap L^2(0,T;H^1(M))$ of the form
\begin{equation}
\label{eq:CGO_decreasing_with_conformal}
u_1(t,x) = \conf e^{-s^2 \beta^2 t -  s x_1}  (v_s+r_1).
\end{equation}
There also exists a solution $u_2 \in H^1(0,T;H^{-1}(M))\cap L^2(0,T;H^1(M))$ to the equation $\mathcal{L}_{c,g,A,q} u_2 =0$ in $Q$ of the form
\begin{equation}
\label{eq:CGO_increasing_with_conformal}
u_2(t,x) = \conf e^{s^2 \beta^2 t +  s x_1}  (w_s+r_2).
\end{equation}
Here $s=\frac{1}{h}$, $v_s,w_s \in C^\infty(Q)$ are the Gaussian beam quasimodes obtained in Theorem \ref{prop:Gaussian_beam}, and the correction terms $r_j$, $j=1,2$ satisfy the estimate
\begin{equation}
\label{eq:est_remainders}
\|r_j\|_{L^2(0,T;H^1_\scl(M))}=o(h), \quad h\to 0.
\end{equation}
\end{thm}

\section{Proof of Theorem \ref{thm:uniqueness}}
\label{sec:proof}

In this section, we utilize the CGO solutions constructed in Section \ref{sec:CGO_solution} to establish Theorem \ref{thm:uniqueness}. The argument consists of several steps. We first derive an integral identity corresponding to the partial input-output operator \eqref{eq:bdy_measurement}, followed by substituting the CGO solutions into this identity. These solutions will reduce the recovery of unknown coefficients to the injectivity of attenuated geodesic ray transforms, from which we apply Assumption \ref{asu:inj} to complete the proof. To simplify the notations, in what follows we shall denote $A:=A^{(1)}-A^{(2)}$ and $q:=q_1-q_2$.

\subsection{Recovery of the convection term}
\label{subsec:proof_A}

We begin this subsection by deriving an integral identity. Let $u_1\in H^1(0,T;H^{-1}(M)) \cap L^2(0,T;H^1(M))$ be an exponentially growing solution to the equation $\mathcal{L}_{c,g,A^{(1)},q_1}^\ast u_1=0$ in $Q$, and let $u_2\in H^1(0,T;H^{-1}(M)) \cap L^2(0,T;H^1(M))$ be a solution to the initial boundary value problem
\begin{equation}
\label{eq:ibvp_proof}
\begin{cases}
\mathcal{L}_{c,g,A^{(2)},q_2}u_2(t,x)=0 & \text{ in } Q,
\\
u(t,x) = f_1(t,x)  &\text{ on } \Sigma,
\\
u(0,\cdot)=f_2(x) &\text{ in } M.
\end{cases}
\end{equation}
Due to the assumption $\Lambda_{A^{(1)},q_1}=  \Lambda_{A^{(2)},q_2}$, there exists a function $v$ such that $\mathcal{L}_{c,g,A^{(1)},q_1} v=0$ in $Q$ and
\[
(u_2-v)|_{\Sigma} = (u_2-v)|_{t=0} = (u_2-v)|_{t=T}=
\p_\nu (u_2-v)|_{V}
=0.
\]
Then the function $u:=v-u_2$ satisfies the following initial boundary value problem:
\begin{equation}
\label{eq:ibvp_difference}
\begin{cases}
\mathcal{L}_{c,g,A^{(1)},q_1}u(t,x)
=
2\n{A, du_2}_g + \left(-(d^\ast A)-q+|A^{(1)}|^2_g-|A^{(2)}|^2_g\right)u_2 & \text{ in } Q,
\\
u(t,x) = 0  &\text{ on } \Sigma,
\\
u(0,\cdot)=0 &\text{ in } M.
\end{cases}
\end{equation}
Since $2\n{A, du_2}_g + \left(-(d^\ast A)-q+|A^{(1)}|^2_g-|A^{(2)}|^2_g\right)u_2\in L^2(Q)$, by \cite[Theorem 1.43]{Choulli_book}, the problem \eqref{eq:ibvp_difference} admits a unique  solution $u\in H^1(0,T;L^2(M)) \cap L^2(0,T;H^2(M))$. Furthermore, it holds that $\p_\nu u \in L^2(0,T;H^{1/2}(\Sigma))$. From here, we argue similarly as in \cite[Section 4]{Mishra_Purohit_Vashisth} to obtain the integral identity
\begin{equation}
\label{eq:int_identity}
2\int_Q \n{A, d u_2}_g \overline{u}_1 dV_gdt +
\int_Q (|A^{(1)}|^2_g - |A^{(2)}|^2_g - (d^\ast A) - q) u_2 \overline{u}_1dV_gdt
=
-\int_{\Sigma \setminus V} \overline u_1 \p_\nu u dS_gdt.
\end{equation}

We next substitute the CGO solution constructed in Section \ref{sec:CGO_solution} into the integral identity \eqref{eq:int_identity}, multiply both sides by $h$, and analyze the limit as $h \to 0$. Let us first show that
\begin{equation}
\label{eq;est_rhs_integral}
\int_{\Sigma \setminus V} \overline u_1 \p_\nu u dS_gdt
=
\begin{cases}
o(h^{-1}) & \text{ if } A\ne 0,
\\
o(1) & \text{ if } A= 0,
\end{cases}
\quad h\to 0,
\end{equation}
by arguing similarly as in the proof of \cite[Lemma 5.1]{Liu_Saksala_Yan}. We shall present the proof for the sake of completeness.

For all $\varepsilon>0$, let
\[
\p M_{+,\varepsilon} = \{x\in \p M: \p_\nu \Gamma(x) >\varepsilon\}
\quad \text{ and } \quad
\Sigma_{+,\varepsilon} = (0,T) \times \p M_{+,\varepsilon},
\]
where $\Gamma(x)=x_1$.
Due to the compactness of the set $\{x\in \p M: \p_\nu\Gamma (x) =0\}$, there exists $\varepsilon>0$ such that $\Sigma \setminus V \subset \Sigma_{+,\varepsilon}$. Therefore,
by the Cauchy-Schwarz inequality, together with the estimates \eqref{eq:estimate_v}, \eqref{eq:estimate_w}, and \eqref{eq:est_remainders}, we have
\begin{align*}
\left|\int_{\Sigma \setminus V} \overline u_1 \p_\nu u dS_gdt\right|
&\le
\int_{\Sigma_{+,\varepsilon}} e^{-s^2 \beta^2 t -  s x_1} |\p_\nu u| (|v_s|+|r_1|) dS_g dt
\\
& \le
C\left(\int_{\Sigma_{+,\varepsilon}} |e^{-s^2 \beta^2 t -  s x_1} \p_\nu u|^2 dS_gdt \right)^{1/2} \left(\|v_s\|_{L^2(\Sigma_{+,\varepsilon})}+\|r_1\|_{L^2(\Sigma)}\right).
\end{align*}

Let us estimate the quantities in the inequality above. To this end, we deduce that
\begin{align*}
\left(\int_{\Sigma_{+, \varepsilon}}|\p_\nu u e^{-s^2 \beta^2 t -  s x_1}|^2   dS_gdt\right)^{1/2}
&= \frac{1}{\sqrt \varepsilon}\left(\int_{\Sigma_{+, \varepsilon}} \varepsilon | e^{-s^2 \beta^2 t -  s x_1} \p_\nu u|^2   dS_gdt\right)^{1/2}
\\
&\le \frac{1}{\sqrt \varepsilon} \left(\int_{\Sigma_{+, \varepsilon}} \p_\nu \Gamma | e^{-s^2 \beta^2 t -  s x_1} \p_\nu u|^2   dS_gdt\right)^{1/2}
\\
&\le \frac{1}{\sqrt \varepsilon}\left(\int_{\Sigma_{+}} \p_\nu \Gamma |e^{-s^2 \beta^2 t -  s x_1} \p_\nu u|^2dS_gdt \right)^{1/2},
\end{align*}
where   $\Sigma_+=(0,T)\times \p M_+^{\text{int}}$.
From here, an application of the boundary Carleman estimate \eqref{eq:boundary_Carleman} gives us
\[
\left(\int_{\Sigma_{+}} \p_\nu \Gamma | e^{-s^2 \beta^2 t -  s x_1} \p_\nu u|^2dS_gdt\right)^{1/2}
\le
\mathcal{O}(h^{1/2})\frac{1}{\sqrt \varepsilon} \|e^{-s^2 \beta^2 t -  s x_1} \mathcal{L}_{c,g,A^{(1)},q_1}u \|_{L^2(Q)}.
\]
In view of the initial boundary value problem \eqref{eq:ibvp_difference}, we apply the estimates \eqref{eq:estimate_w} and \eqref{eq:est_remainders} to obtain
\[
\|e^{-s^2 \beta^2 t -  s x_1} \mathcal{L}_{c,g,A^{(1)},q_1}u \|_{L^2(Q)} =
\begin{cases}
\cO(h^{-1}) & \text{ if } A\ne 0,
\\
\cO(1) & \text{ if } A= 0.
\end{cases}
\]
Thus, it follows immediately that
\begin{equation}
\label{eq:est_bdy_integral}
\left(\int_{\Sigma_{+}} \p_\nu \Gamma | e^{-s^2 \beta^2 t -  s x_1} \p_\nu u|^2   dS_gdt\right)^{1/2}
=
\begin{cases}
\cO(h^{-1/2}) & \text{ if } A\ne 0,
\\
\cO(h^{1/2}) & \text{ if } A= 0.
\end{cases}
\end{equation}

On the other hand, by the estimate \eqref{eq:est_remainders}, as well as the inequalities
\[
\|r_1\|_{L^2(\Sigma)}\le \|r_1\|_{L^2(Q)}^{1/2}\|r_1\|_{H^1(Q)}^{1/2}  \quad \text{and} \quad  \|r_1\|_{H^1(Q)}\le Ch^{-1}\|r_1\|_{H^1_\scl(Q)},
\]
we have
\[
\|r_1\|_{L^2(\Sigma)} =o(h^{-1/2}), \quad h \to 0.
\]
Moreover, it follows from the arguments given in the proof of \cite[Theorem 6.2]{Cekic} and \cite[Lemma 5.1]{Liu_Saksala_Yan} that
\[
\|v_s\|_{L^2(\Sigma_{+,\varepsilon})}=\cO(1), \quad h \to 0.
\]
Hence, we get from the previous two estimates that
\begin{equation}
\label{eq:est_sum_bdy}
\|v_s\|_{L^2(\Sigma_{+,\varepsilon})}+ \|r_1\|_{L^2(\Sigma)}  = o(h^{-1/2}), \quad h\to 0.
\end{equation}
Finally, we obtain the estimate \eqref{eq;est_rhs_integral} by combining \eqref{eq:est_bdy_integral} and \eqref{eq:est_sum_bdy}.

Turning our attention to the left-hand side of the integral identity \eqref{eq:int_identity}, using the CGO solutions \eqref{eq:CGO_decreasing_with_conformal} and \eqref{eq:CGO_increasing_with_conformal} for $u_1$ and $u_2$, we compute that
\[
u_2 \overline{u}_1
=
c^{-\frac{n-2}{2}} (\overline{v}_sw_s+\overline{v}_s r_2+ w_s\overline{r}_1+r_2 \overline{r}_1).
\]
By the Cauchy-Schwarz inequality, combined with the estimates \eqref{eq:estimate_v}, \eqref{eq:estimate_w}, and \eqref{eq:est_remainders}, we deduce that
\[
\|u_2 \overline{u}_1\|_{L^1(Q)}\leq \cO(1).
\]
Thus, in view of the assumption that $A\in W^{1,\infty}(Q)$ and $q\in C(\overline{Q})$, we conclude that
\begin{equation}
\label{eq:est_lower_order}
h\left|\int_{Q}(|A^{(1)}|^2_g - |A^{(2)}|^2_g - (d^\ast A) - q) u_2 \overline{u}_1   dV_gdt \right|=\cO(h), \quad h\to 0.
\end{equation}

We next calculate that
\begin{equation}
\label{eq:computation_du2}
\begin{aligned}
h &\int_{Q}2 \n{A, d u_2}_g\overline{u_1}   dV_g dt
\\
= &
h\int_{Q}2 \n{A, d (c^{-(n-2)/4})}_g c^{-\frac{n-2}{4}} (\overline{v}_s w_s + \overline{v}_s r_2+ w_s\overline{r}_1 +\overline{r}_1 r_2 )   dV_g dt
\\
&+
h\int_{Q}2 \n{A, d \left({s^2\beta^2t+s x_1}\right)}_g c^{-\frac{n-2}{2}}  (\overline{v}_s w_s + \overline{v}_s r_2+ w_s\overline{r}_1 +\overline{r}_1 r_2 )    dV_g dt
\\
&+
h\int_{Q}2 \n{A, d (w_s+r_2)}_g c^{-\frac{n-2}{2}}   (\overline{v}_s + \overline{r}_1 )   dV_g dt
\\
:= &
I_1+I_2+I_3.
\end{aligned}
\end{equation}
Then it follows from the estimates \eqref{eq:estimate_v}, \eqref{eq:estimate_w}, and \eqref{eq:est_remainders} that
\begin{equation}
\label{eq:est_I1}
\left|I_1\right|= \cO\left(h\right), \quad h\to 0.
\end{equation}

We now move to investigate $I_2$. To this end, arguing similarly as above, we deduce that
\begin{equation}
\label{eq:est_real_part_I2}
\left| h\int_{Q} \frac{2}{h} \n{A, d {x_1}}_g c^{-(n-2)/2}   ( \overline{v}_s r_2+ w_s\overline{r}_1 +\overline{r}_1 r_2 )   dV_g dt\right|= o(h), \quad h \to 0.
\end{equation}
To analyze the term involving $\overline{v}_sw_s$, we first note that the identity $g=c(e \oplus g_0)$ gives us $\n{A,d{x_1}}_g= c^{-1}A_1$ and $dV_g = c^{\frac{n}{2}} dV_{g_0} dx_1$. Let us also continuously extend $A$ by zero outside $Q$ and denote the extension by the same letter. This is possible since $A^{(1)},A^{(2)}\in W^{1,\infty}(Q)$, and they agree on $\p Q$.  Therefore, using Fubini’s theorem, in conjunction with the limit  \eqref{eq:limit_prod_vw} and the
dominated convergence theorem, we obtain the following limit as $h \to 0$:

\begin{equation}
\label{eq:convergence_quasimode}
\begin{aligned}
&2h\int_{Q}\frac{1}{h} \n{A, d {x_1}}_g  c^{-(n-2)/2}    \overline{v}_s w_s   dV_{g} dt
\\
&\to 2 (1-\beta^2)^{-\frac{n-6}{4}}\int_{\mathbb{R}}\int_{\mathbb{R}}   \int_{0}^{\frac{L}{\sqrt{1-\beta^2}}} A_1(t,x_1,\gamma(\sqrt{1-\beta^2}r))
e^{\overline{\Phi_1}(t, x_1, r)+\Phi_2(t, x_1, r)}
\eta(t, x_1,r)   dr dx_1 dt,
\end{aligned}
\end{equation}
where the functions  $\Phi_1$ and $\Phi_2$ satisfy the following transport equations, respectively:
\begin{equation}
\label{eq:transport_eqn_Phi_1}
(\p_{x_1}-i\p_r) \Phi_{1}(t, x_1, r) = \overline A^{(1)}_1 (t,x_1,\sqrt{1-\beta^2} r, 0) -i  \sqrt{1 - \beta^2} \overline A^{(1)}_\tau (t, x_1, \sqrt{1-\beta^2}r, 0),
\end{equation}
and
\begin{equation}
\label{eq:transport_eqn_Phi_2}
(\p_{x_1}+i\p_r) \Phi_{2}(t, x_1, r)=-A^{(2)}_1 (t,x_1, \sqrt{1-\beta^2}r, 0) - i \sqrt{1 - \beta^2} A^{(2)}_\tau (t, x_1,\sqrt{1-\beta^2} r, 0).
\end{equation}

Finally, we analyze $I_3$ as $h \to 0$. By the estimates \eqref{eq:estimate_v}, \eqref{eq:estimate_w}, and \eqref{eq:est_remainders}, we deduce that
\begin{equation}
\label{eq:est_I3_part1}
\left|h\int_{Q}2 \left(\n{A, d w_s}_g\overline{r}_1 + \n{A, d r_2}_g(\overline{v}_s +\overline{r}_1 )\right)  c^{-(n-2)/2}      dV_g dt \right|= o(h), \quad h\to 0.
\end{equation}
On the other hand, by setting $\tilde{g}= e\oplus g_0$, we see that $\n{A,dw_s}_g= c^{-1}\n{A,dw_s}_{\tilde{g}}$.
Combining this identity with the fact that $A=0$ outside $Q$, Fubini’s theorem, the concentration property \eqref{eq:concentration_one_form_ws}, and the
dominated convergence theorem, we have
\begin{equation}
\label{eq:limit_I3}
\begin{aligned}
&2h\int_{Q} \n{A, d w_s}_g c^{-(n-2)/2}   \overline{v}_s   dV_g dt \\
&  \to 2i (1-\beta^2)^{-\frac{n-8}{4}}\int_{\mathbb{R}}\int_{\mathbb{R}}   \int_{0}^{\frac{L}{\sqrt{1-\beta^2}}}  A_\tau (\dot \gamma(\sqrt{1-\beta^2}r))  e^{\overline{\Phi_1}(t, x_1, r)+\Phi_2(t, x_1, r)}
\eta(t, x_1,r)   dr dx_1 dt
\end{aligned}
\end{equation}
in the limit $h\to 0$, where $\Phi_1$ and $\Phi_2$ satisfy the transport equations \eqref{eq:transport_eqn_Phi_1} and \eqref{eq:transport_eqn_Phi_2}, respectively. Since $\beta \in (\frac{1}{\sqrt{3}},1)$, we conclude from    \eqref{eq:est_lower_order}--\eqref{eq:convergence_quasimode}, \eqref{eq:est_I3_part1}, and \eqref{eq:limit_I3} that
\begin{equation}
\label{eq:identity_after_CGO}
\begin{aligned}
&   \int_{\mathbb{R}}\int_{\mathbb{R}}   \int_{0}^{\frac{L}{\sqrt{1-\beta^2}}} \left(A_1(t,x_1,\gamma(\sqrt{1-\beta^2}r))+i\sqrt{1-\beta^2}A_\tau(t,x_1,\gamma(\sqrt{1-\beta^2}r) \right)
\\
& \qquad \qquad \qquad    \times
e^{\overline \Phi_1(t, x_1, r)+\Phi_2(t, x_1, r)}\eta(t, x_1,r)   dr dx_1 dt=0.
\end{aligned}
\end{equation}

By taking the conjugate of \eqref{eq:transport_eqn_Phi_1} and add it to \eqref{eq:transport_eqn_Phi_2}, we get
\begin{equation}
\label{difference_phi_equation}
\overline \partial\left(\overline \Phi_{1}+\Phi_{2}\right) = A_1 (t,x_1,\sqrt{1-\beta^2} r, 0) +i  \sqrt{1 - \beta^2}  A_\tau  (t, x_1, \sqrt{1-\beta^2}r, 0),
\end{equation}
where $\overline{\p} = \p_{x_1}+i\p_r$.
Thus, \eqref{eq:identity_after_CGO} and \eqref{difference_phi_equation} imply that
\[
\int_{\mathbb{R}}\int_{\mathbb{R}} \int_{0}^{\frac{L}{\sqrt{1-\beta^2}}} \overline \partial\left(\overline \Phi_{1}+\Phi_{2}\right) e^{\overline \Phi_1(t, x_1, r)+\Phi_2(t, x_1, r)}\eta(t, x_1,r)   dr dx_1 dt=0.
\]
Let us recall that $\overline{\p} \eta =0$. Hence, the product rule yields that
\[
\int_{\mathbb{R}}\int_{\mathbb{R}} \int_{0}^{\frac{L}{\sqrt{1-\beta^2}}} \overline \partial\left(e^{\overline \Phi_1(t, x_1, r)+\Phi_2(t, x_1, r)} \eta(t, x_1,r)\right)    dr dx_1 dt=0.
\]

We now choose $\eta(t,x_1,r) =  e^{-i \sqrt{1-\beta^2}\mu(x_1+ir)} \eta_2(t)$, where $\mu \in \mathbb{R}$ is   fixed, and the function $\eta_2 \in C_0^\infty(\R)$ is arbitrary. Then it follows immediately that  $\overline{\p} e^{-i \sqrt{1-\beta^2}\mu(x_1+ir)}=0$ and $\overline{\p} \eta=0$. Following the arguments in  \cite{Krupchyk_Uhlmann_magschr,Liu_Saksala_Yan, Mishra_Purohit_Vashisth}, which utilize the compact support of $A$ in the $x_1$-variable and  Stokes' theorem,  we remove the factor $e^{\overline{\Phi}_1+\Phi_2}$  to obtain
\begin{align*}
&\int_{\mathbb{R}}\int_{\mathbb{R}} \int_{0}^{\frac{L}{\sqrt{1-\beta^2}}} e^{-i \sqrt{1-\beta^2}\mu(x_1+ir)} \left(A_1 (t,x_1,\sqrt{1-\beta^2} r, 0) +i  \sqrt{1 - \beta^2}  A_\tau (t, x_1, \sqrt{1-\beta^2}r, 0)\right)
\\
& \qquad \qquad \qquad    \times  \eta_2(t)   dr dx_1 dt=0.
\end{align*}
Since $\eta_2\in C^\infty(\mathbb{R})$ is arbitrary, we have for almost every $t\in \mathbb{R}$ that
\[
\int_{\mathbb{R}} \int_{0}^{\frac{L}{\sqrt{1-\beta^2}}} e^{-i\sqrt{1-\beta^2}\mu(x_1+ir)}\left(A_1 (t,x_1,\sqrt{1-\beta^2} r, 0) +i  \sqrt{1 - \beta^2}  A_\tau (t, x_1, \sqrt{1-\beta^2}r, 0)\right)    dr dx_1 =0.
\]
We now recall that $r= \frac{1}{\sqrt{1-\beta^2}}\tau$ and rewrite the previous identity as
\[
\int_{0}^{L} e^{\mu  \tau}\left[f(\gamma(\tau))+i\sqrt{1-\beta^{2}} F(\dot{\gamma}(\tau))\right] d\tau = 0,
\]
where
\begin{equation}
\label{eq:def_f_and_F}
\begin{aligned}
&f(t, x', \mu)= \int_\R   e^{-i\sqrt{1-\beta^2}\mu x_{1}}  A_1(t, x_1, x')  dx_1,
\\
&F(x', \mu) = \sum_{j=2}^{n}\left(\int_\R   e^{-i\sqrt{1-\beta^2}\mu x_{1}} A_j(t,x_1,x') dx_1\right) dx^j.
\end{aligned}
\end{equation}
In what follows, we denote
\[
F_j (x', \mu) := \int_\R   e^{-i\sqrt{1-\beta^2}\mu x_{1}}  A_j(x_1,x') dx_1.
\]
Since $A\in W^{1,\infty}(Q)$, it holds that  $f (\cdot, \mu)\in W^{1,\infty}(M_0)$ and $F(\cdot, \mu)\in W^{1,\infty} (M_0, T^\ast M_0)$.

We are now ready to recover $dA$ in $Q$. To this end, we choose  $\mu\in \R$ such that $|\mu| \le \varepsilon$ and apply Assumption \ref{asu:inj} to the pair $(f, i\sqrt{1-\beta^2}F)$ with constant attenuation $\alpha =\mu$ to obtain
\begin{equation}
\label{eq:injectivity_conclusion}
f = \mu p \quad \text{ and } \quad F = -  \frac{i}{\sqrt{1-\beta^2}} d_{x'} p
\end{equation}
for some function $p \in W^{2,\infty} (M_0)$ such that $p|_{\partial M_0}=0$. Thus, we may express $F$ as
\[
F(x', \mu) = \sum_{j=2}^{n}\mathcal{F}_{x_1} A_j(t,\cdot,x') (\sqrt{1-\beta^2}\mu)   dx^j =   \sum_{j=2}^{n} -  \frac{i}{\sqrt{1-\beta^2}}  \left(\partial_j p \right)  dx^j,
\]
where $ \mathcal{F}_{x_1}(\cdot)$ denotes the Fourier transform with respect to $x_1$.

By a direct computation, we have
$\cF_{x_1} [\p_k A_j-\p_j A_k](\sqrt{1-\beta^2}\mu, \cdot) =0$, $j,k=2,\dots, n$,
for all sufficiently small $|\mu|$. Since $A$ is compactly supported in
the $x_1$-variable, the Paley-Weiner theorem implies that  $\cF_{x_1} [\p_k A_j-\p_j A_k](\sqrt{1-\beta^2}\mu, \cdot)$ extends to an entire function of $\mu$. Therefore, we get
\[
\cF_{x_1} [\p_k A_j-\p_j A_k](\sqrt{1-\beta^2}\mu, \cdot) =0, \quad \mu \in \R.
\]
Due to the injectivity of the Fourier transform, it holds that
\begin{equation}
\label{eq:dA_vanish_jge2}
\partial_k A_j-\partial_j A_k =0, \quad j,k = 2,\dots,n.
\end{equation}
We next compute $\cF_{x_1} \left[\partial_j A_1-\partial_1 A_j\right](\sqrt{1-\beta^2}\mu, \cdot)$.
By the definition of $f$ and \eqref{eq:injectivity_conclusion}, we have for all $2 \leq j \leq n$ that
\[
\mu\partial_j p = \partial_jf =\int_\R  e^{-i\sqrt{1-\beta^2}\mu x_{1}}\partial_j A_1(x_1,x')  dx_1,
\]
On the other hand, as $F_j$ is independent of $x_1$, we get from the product rule that
\[
0=  \int_\R   \p_1 \left(e^{-i\sqrt{1-\beta^2}\mu x_{1}}  A_j(x_1,x') \right)  dx_1 =  -i\sqrt{1-\beta^2}\mu F_j+\int_\R  e^{-i\sqrt{1-\beta^2} \mu x_{1}} \partial_1 A_j(x_1,x')   dx_1.
\]
Using the identity $i\mu \sqrt{1-\beta^2} F_j = \mu\partial_j p$, which follows immediately from \eqref{eq:injectivity_conclusion}, we get from the previous two equations that
\begin{align*}
\cF_{x_1} \left[\partial_j A_1-\partial_1 A_j\right](\sqrt{1-\beta^2}\mu, \cdot)=0.
\end{align*}
Arguing similarly as above, we have
\begin{equation}
\label{eq:dA_vanish_x1}
\partial_j A_1-\partial_1 A_j =0.
\end{equation}
From here, we combine \eqref{eq:dA_vanish_jge2} and \eqref{eq:dA_vanish_x1} to conclude that $dA=0$ in $Q$.

We next follow the arguments from \cite[Section 5]{Feizmohammadi_et_all_2019} to complete the recovery of $A(t,x)$ up to the gauge described in \eqref{eq:gauge}, see also  \cite[Section 2]{Selim_Yan} and \cite[Section 4]{Yan}. Let
\begin{equation}
\label{eq:def_Phi}
\Psi(t,x_1,x') = \int_{-a}^{x_1} A_1 (t,y_1,x')  dy_1,
\end{equation}
where $\supp A_1(t,\cdot, x') \subset (-a,a)$. Using \eqref{eq:def_f_and_F}, we have at $\mu=0$ that
\[
0= f(t,x',0) = \int_\R A_1(t,x_1, x')  dx_1,
\]
which implies that $\Psi$ has compact support in the $x_1$-variable. Hence, the Fourier transform of $\Psi$ with respect to the $x_1$ variable, which we denote as $\cF_{x_1}\Psi(t,x', \xi)$, extends to an entire function of $\xi$. Since $\p_{x_1} \Psi = A_1$, we utilize \eqref{eq:injectivity_conclusion} to deduce that
\[
i \sqrt{1-\beta^2}\mu \cF_{x_1}\Psi(t,x', \sqrt{1-\beta^2}\mu) =f(t,x', \mu) = \mu p(t,x').
\]
Therefore, we have for all $\mu \ne 0$ that
\[
\cF_{x_1}\Psi(t,x', \sqrt{1-\beta^2}\mu) = -\frac{i}{\sqrt{1-\beta^2}} p(t,x').
\]
Applying \eqref{eq:def_f_and_F} and \eqref{eq:injectivity_conclusion} again,   the identity
\[
d_{x'} \cF_{x_1} \Psi(t,x', \sqrt{1-\beta^2}\mu)
= -\frac{i}{\sqrt{1-\beta^2}} d_{x'} p(t,x') = F(t, x', \mu)
= \sum_{j=2}^n \cF_{x_1} A_j (t,\sqrt{1-\beta^2}\mu, x')  dx^j.
\]
holds for all sufficiently small $\mu \ne 0$. By the analyticity of the Fourier transform, the equation above holds for all $\mu \in \R$. Applying the inverse Fourier transform with respect to the $x_1$-variable, we obtain
\[
d_{x'} \Psi = \sum_{j=2}^n A_jdx^j.
\]
Moreover, as  $\p_{x_1}\Psi = A_1$, we conclude that $d_{x} \Psi = A$.

Since $\p M$ is connected and $d \Psi = A =0$ on $\p M$, it holds that $\Psi(t,\cdot)$ is a constant on $\p M$ for almost every $t\in(0,T)$. Thus, we may modify $\Psi$ by a constant and assume that $\Psi =0$ on $\Sigma$. Moreover, since $A^{(1)}=A^{(2)}$ on $\p Q$, we have $d\Psi(0,\cdot) = d \Psi(T,\cdot)=0$. Hence, $\Psi(0,\cdot)$ and $ \Psi(T,\cdot)$ are both constant in $M$. Furthermore, their traces on $\p M$ vanish, we conclude that $\Psi(0,\cdot)=  \Psi(T,\cdot)=0$, giving us $\Psi|_{\p Q}=0$. This completes the proof of the recovery of $A(t,x)$ up to the gauge invariance described in \eqref{eq:gauge}.

\subsection{Recovery of the density coefficient}
\label{subsec:density_recovery}

We now establish the uniqueness of the density coefficient $q(t,x)$ up to the natural gauge. To that end, we define $A^{(3)}  := A^{(1)} - d\Psi$ and $q_3  = q_1  - \partial_t\Psi$. Let us recall that we have established $A^{(2)}=A^{(1)} - \nabla_g\Psi$ in Subsection \ref{subsec:proof_A}. Thus, it follows immediately that $A^{(3)} = A^{(2)}$ in $Q$. Furthermore, by  \cite[Proposition 1.3]{Mishra_Purohit_Vashisth} and the hypothesis $\Lambda_{A^{(1)},q_1}=  \Lambda_{A^{(2)},q_2}$,  we get $\Lambda_{A^{(3)},q_3} = \Lambda_{A^{(2)},q_2}$. Then the integral identity \eqref{eq:int_identity} becomes
\begin{equation}
\label{eq:int_id_density}
\int_{Q} (q_2 - q_3) u_2  {\overline u}_1 dV_g  dt = -\int_{\Sigma \setminus V} \overline u_1 \p_\nu u dS_gdt.
\end{equation}
In view of \eqref{eq;est_rhs_integral}, we obtain
\[
\int_{\Sigma \setminus V} \overline u_1 \p_\nu u dS_gdt = o(1), \quad h\to 0.
\]
Therefore, it follows that
\[
\int_{Q} (q_2 - q_3) u_2  {\overline u}_1 dV_g  dt
\to 0, \quad h\to 0.
\]

We next substitute the CGO solutions $u_1$ and $u_2$ given by \eqref{eq:CGO_decreasing_with_conformal} and \eqref{eq:CGO_increasing_with_conformal} into \eqref{eq:int_id_density}.  Using the Cauchy-Schwarz inequality,  together with the estimates \eqref{eq:estimate_v}, \eqref{eq:estimate_w}, and \eqref{eq:est_remainders}, we get
\[
\int_{Q} (q_2 - q_3) c^{-(n-2)/2} w_s\overline v_s dV_{g}dt \to 0, \quad h\to 0.
\]
Arguing similarly as in Subsection \ref{subsec:proof_A}, we have
\[
\int_{\mathbb{R}}\int_{\mathbb{R}}   \int_{0}^{\frac{L}{\sqrt{1-\beta^2}}} (q_2-q_3)(t,x_1,\gamma(\sqrt{1-\beta^2}r))  e^{\overline{\Phi_1}(t, x_1, r)+\Phi_2(t, x_1, r)}
\eta(t, x_1,r)dr dx_1 dt =0.
\]

By proceeding similarly to the argument from \eqref{eq:identity_after_CGO} onward, we conclude that $q_2 =  q_3$ in $Q$. Since $q_3 =  q_{1} -  \partial_t \Psi$, we obtain $q_{1}  - q_{2} = \partial_t\Psi $ in $Q$. This completes the proof of Theorem \ref{thm:uniqueness}.

\bibliographystyle{abbrv}
\bibliography{bib_convection_diffusion}

\begin{thebibliography}{10}

\bibitem{Avdonin_Seidman}
S.~Avdonin and T.~Seidman.
\newblock Identification of $q(x)$ in $u_t=\delta u - qu$ from boundary observations.
\newblock {\em SIAM Journal on Control and Optimization}, 33(4):1247--1255, 1995.

\bibitem{Babich_Buldyrev}
V.~Babich and V.~Buldyrev.
\newblock {\em Short-Wavelength Diffraction Theory: Asymototpic Methods.}, volume~4 of {\em Springer Series on Wave Phenomena}.
\newblock Springer, Berlin, 1991.

\bibitem{Babich_Ulin}
V.~Babich and V.~Ulin.
\newblock The complex space-time ray method and quasi-photons.
\newblock {\em Mathematical problems in the theory of wave propagation, 12}, 117:5--12, 1981.

\bibitem{Belishev_Katchalov}
M.~Belishev and A.~Katchalov.
\newblock Boundary control and quasiphotons in a {R}iemannian manifold reconstruction problem via dynamical data (in {R}ussian).
\newblock {\em Zap. Nauchn. Sem. Leningrad. Otdel. Mat. Inst. Steklov.}, 203:21--50, 1992.

\bibitem{Bellassoued_Fraj}
M.~Bellassoued and O.~{B. Fraj}.
\newblock Stability estimates for time-dependent coefficients appearing in the magnetic {S}chr{\"o}dinger equation from arbitrary boundary measurements.
\newblock {\em Inverse Problems and Imaging}, 2020.

\bibitem{Bellassoued_Rassas}
M.~Bellassoued and I.~Rassas.
\newblock Stability estimate for an inverse problem of the convection-diffusion equation.
\newblock {\em Journal of Inverse and Ill-posed Problems}, 28(1):71--92, 2020.

\bibitem{Canuto_Kavian}
B.~Canuto and O.~Kavian.
\newblock Determining coefficients in a class of heat equations via boundary measurements.
\newblock {\em SIAM Journal on Mathematical Analysis}, 32(5):963--986, 2001.

\bibitem{Caro_Kian}
P.~Caro and Y.~Kian.
\newblock Determination of convection terms and quasilinearities appearing in diffusion equations.
\newblock {\em preprint, Arxiv: 1812.08495}, 2018.

\bibitem{Cekic}
M.~Ceki\'c.
\newblock The {C}alder\'on problem for connections.
\newblock {\em Communications in Partial Differential Equations}, 42(11):1781--1836, 2017.

\bibitem{Cheng_Yamamoto}
J.~Cheng and M.~Yamamoto.
\newblock Identification of convection term in a parabolic equation with a single measurement.
\newblock {\em Nonlinear Analysis}, 50(2):163--171, 2002.

\bibitem{Choulli_book}
M.~Choulli.
\newblock {\em Une introduction aux probl{\`e}mes inverses elliptiques et paraboliques}, volume~65 of {\em Math{\'e}matiques et Applications}.
\newblock Springer, Berlin, 2009.

\bibitem{Choulli_Kian}
M.~Choulli and Y.~Kian.
\newblock Logarithmic stability in determining the time-dependent zero order coefficient in a parabolic equation from a partial {D}irichlet-to-{N}eumann map. {A}pplication to the determination of a nonlinear term.
\newblock {\em Journal de Math{\'e}matiques Pures et Appliqu{\'e}es}, 114(235-261), 2018.

\bibitem{deHoop_Ilmavirta}
M.~{de Hoop} and J.~Ilmavirta.
\newblock Abel transforms with low regularity with applications to x-ray tomography on spherically symmetric manifolds.
\newblock {\em Inverse Problems}, 33:124033, 2017.

\bibitem{Deng_Yu_Yang}
Z.~C. Deng, J.~N. Yu, and L.~Yang.
\newblock Identifying the coefficient of first-order in parabolic equation from final measurement data.
\newblock {\em Mathematics and Computers in Simulation}, 77(4):421--435, 2008.

\bibitem{Ferreira_Kenig_Salo_Uhlmann}
D.~{Dos Santos Ferreira}, C.~Kenig, M.~Salo, and G.~Uhlmann.
\newblock Limiting {C}arleman weights and anisotropic inverse problems.
\newblock {\em Inventiones {M}athematicae}, 178:119--171, 2009.

\bibitem{DDS_Kenig_Sjo_Uhl}
D.~{Dos Santos Ferreira}, C.~Kenig, J.~Sj\"ostrand, and G.~Uhlmann.
\newblock Determining a magnetic {S}chr\"odinger operator from partial {C}auchy data.
\newblock {\em Communications in Mathematical Physics}, 271(2):467--488, 2007.

\bibitem{Ferreira_Kur_Las_Salo}
D.~{Dos Santos Ferreira}, Y.~Kurylev, M.~Lassas, and M.~Salo.
\newblock The {C}alder\'on problem in transversally anisotropic geometries.
\newblock {\em Journal of European Mathematical Society}, 18(11):2579--2626, 2016.

\bibitem{Evans}
L.~C. Evans.
\newblock {\em Partial Differential Equations}, volume~19 of {\em Graduate Studies in Mathematics}.
\newblock American Mathematical Society, Providence, RI, 2010.

\bibitem{Feizmohammadi_heat}
A.~Feizmohammadi.
\newblock An inverse boundary value problem for isotropic nonautonomous heat flows.
\newblock {\em Mathematische Annalen}, 388:1569--1607, 2024.

\bibitem{Feizmohammadi_et_all_2019}
A.~Feizmohammadi, J.~Ilmavirta, Y.~Kian, and L.~Oksanen.
\newblock Recovery of time dependent coefficients from boundary data for hyperbolic equations.
\newblock {\em Journal of Spectral Theory}, 11(3):1107--1143, 2021.

\bibitem{Friedlander_Joshi}
G.~Friedlander and M.~Joshi.
\newblock {\em Introduction to the Theory of Distributions}.
\newblock Cambridge University Press, Cambridge, 1998.

\bibitem{Gaitan_Kian}
P.~Gaitan and Y.~Kian.
\newblock A stability result for a time-dependent potential in a cylindrical domain.
\newblock {\em Inverse Problems}, 29(6):065006, 2013.

\bibitem{Isakov_91}
V.~Isakov.
\newblock Completeness of products of solutions and some inverse problems for pde.
\newblock {\em Journal of Differential Equations}, 92(2):305--316, 1991.

\bibitem{Katchalov_Kurylev}
A.~Katchalov and Y.~Kurylev.
\newblock Multidimensional inverse problem with incomplete boundary spectral data.
\newblock {\em Communications in Partial Differential Equations}, 23(1-2):27--59, 1998.

\bibitem{KKL_book}
A.~Katchalov, Y.~Kurylev, and M.~Lassas.
\newblock {\em Inverse Boundary Spectral Problems}, volume 123 of {\em Monographs and Surveys in Pure and Applied Math}.
\newblock Chapman \& Hall/CRC, 2001.

\bibitem{Kenig_Salo}
C.~Kenig and M.~Salo.
\newblock The {C}alder\'on problem with partial data on manifolds and applications.
\newblock {\em Analysis \& PDE}, 6(8):2003--2048, 2013.

\bibitem{Krupchyk_Uhlmann_magschr}
K.~Krupchyk and G.~Uhlmann.
\newblock Inverse problems for magnetic {S}chr{\"o}dinger operators in transversally anisotropic geometries.
\newblock {\em Communications in Mathematical Physics}, 361:525--582, 2018.

\bibitem{Kumar_Purohit}
P.~Kumar and A.~Purohit.
\newblock Inverse boundary value problem for the convection--diffusion equation with local data.
\newblock {\em Applicable Analysis}, 104(11):2195--2204, 2025.

\bibitem{lee2012smooth}
J.~Lee.
\newblock {\em Introduction to Smooth Manifolds}, volume 218 of {\em Graduate Texts in Mathematics}.
\newblock Springer, New York, 2nd edition, 2012.

\bibitem{lionsNonHomogeneousBoundaryValue1972}
J.~L. Lions and M.~Magenes.
\newblock {\em Non-homogeneous Boundary Value Problems and Applications}, volume~I.
\newblock Springer Berlin Heidelberg, Berlin, Heidelberg, 1972.

\bibitem{Liu_Saksala_Yan}
B.~Liu, T.~Saksala, and L.~Yan.
\newblock Partial data inverse problem for hyperbolic equation with time-dependent damping coefficient and potential.
\newblock {\em SIAM Journal on Mathematical Analysis}, 56(4):5678--5722, 2024.

\bibitem{Liu_Saksala_Yan_potential}
B.~Liu, T.~Saksala, and L.~Yan.
\newblock Recovery of a time-dependent potential in hyperbolic equations on conformally transversally anisotropic manifolds.
\newblock {\em Journal of Spectral Theory}, 15(1):123--147, 2025.

\bibitem{Mishra_Purohit_Vashisth}
R.~Mishra, A.~Purohit, and M.~Vashisth.
\newblock Inverse problem for a time-dependent convection-diffusion equation in admissible geometries.
\newblock {\em Research in the Mathematical Sciences}, 12(4):article number 75, 2025.

\bibitem{paternain2019geodesic}
G.~Paternain, M.~Salo, G.~Uhlmann, and H.~Zhou.
\newblock The geodesic {X}-ray transform with matrix weights.
\newblock {\em American Journal of Mathematics}, 141(6):1707--1750, 2019.

\bibitem{Purohit}
A.~Purohit.
\newblock Determining time-dependent convection and density terms in a convection-diffusion equation using partial data.
\newblock {\em Communications on Analysis and Computation}, 3:69--90, 2025.

\bibitem{Ralston_1977}
J.~Ralston.
\newblock Approximate eigenfunctions of the {L}aplacian.
\newblock {\em Journal of Differential Geometry}, 12(1):87--100, 1977.

\bibitem{Ralston_1982}
J.~Ralston.
\newblock Gaussian beams and the propagation of singularity, studies in partial differential equations.
\newblock {\em MAA Studies in Mathematics}, 23:206--248, 1982.

\bibitem{Sahoo_Vashith}
S.~K. Sahoo and M.~Vashisth.
\newblock A partial data inverse problem for the convection-diffusion equation.
\newblock {\em Inverse Problems and Imaging}, 14(1):53--75, 2020.

\bibitem{Selim_Yan}
S.~Selim and L.~Yan.
\newblock Partial data inverse problems for magnetic {S}chr{\"o}dinger operators on conformally transversally anisotropic manifolds.
\newblock {\em Asymptotic Analysis}, 140:25--36, 2024.

\bibitem{Senapati_Vashisth}
S.~Senapati and M.~Vashisth.
\newblock Stability estimate for a partial data inverse problem for the convection-diffusion equation.
\newblock {\em Evolution Equations and Control Theory}, 11(5):1681--1699, 2022.

\bibitem{Stocker}
T.~Stocker.
\newblock {\em Introduction to Climate Modelling}.
\newblock Advances in Geophysical and Environmental Mechanics and Mathematics. Springer, Berlin, Heidelberg, 2011.

\bibitem{Yan}
L.~Yan.
\newblock Inverse boundary problems for biharmonic operators in transversally anisotropic geometries.
\newblock {\em SIAM Journal on Mathematical Analysis}, 53(6):6617--6653, 2021.

\end{thebibliography}
\end{document}